\documentclass[11pt]{article}

\PassOptionsToPackage{numbers, sort&compress}{natbib}

\usepackage[margin=1in]{geometry}
\usepackage{natbib}
\usepackage{libertine}
\usepackage[libertine]{newtxmath}
\usepackage{authblk}

\usepackage[utf8]{inputenc} 
\usepackage[T1]{fontenc}    
\usepackage{hyperref}       
\usepackage{url}            
\usepackage{booktabs}       
\usepackage{amsfonts}       
\usepackage{nicefrac}       
\usepackage{microtype}      

\usepackage[dvipsnames,svgnames,x11names]{xcolor}
\definecolor{figblue}{RGB}{90,190,250}   
\usepackage{tikz}                       
\usetikzlibrary{arrows.meta, positioning}

\usepackage{cleveref}
\usepackage{amsthm,thmtools}
\allowdisplaybreaks
\usepackage{comment}
\usepackage{graphicx}
\usepackage{algorithm}
\usepackage[noend]{algpseudocode}
\usepackage{enumitem}
\setenumerate[1]{itemsep=0pt, partopsep=0pt, parsep=\parskip, topsep=0pt}
\setitemize[1]{itemsep=0pt, partopsep=0pt, parsep=\parskip, topsep=0pt}

\newtheorem{theorem}{Theorem}[section]
\newtheorem{lemma}[theorem]{Lemma}

\newtheorem{definition}{Definition}[section]
\newtheorem{example}[theorem]{Example}
\newtheorem{assumption}{Assumption}[section]

\newtheorem{corollary}[theorem]{Corollary}

\let\originalleft\left
\let\originalright\right
\renewcommand{\left}{\mathopen{}\mathclose\bgroup\originalleft}
\renewcommand{\right}{\aftergroup\egroup\originalright}

\newcommand{\norm}[1]{\left\lVert#1\right\rVert}
\newcommand{\abs}[1]{\left\lvert#1\right\rvert}

\newcommand{\bR}{\mathbb{R}}
\newcommand{\bP}[2][]{\Pr\ifthenelse{\isempty{#1}}{}{_{#1}}\left[#2\right]}
\newcommand{\bE}[2][]{\mathop\mathbb{E}\ifthenelse{\isempty{#1}}{}{_{#1}}\left[#2\right]}
\newcommand{\bI}[2][]{\mathop\mathbb{I}\ifthenelse{\isempty{#1}}{}{_{#1}}\left[#2\right]}
\newcommand{\Var}[2][]{\mathbf{Var}\ifthenelse{\isempty{#1}}{}{_{#1}}\left[#2\right]}

\DeclareMathOperator*{\argmin}{arg\,min}

\newcommand{\horizonvar}{p}
\newcommand{\decayvar}{\lambda}

\newcommand{\yiheng}[1]{\textcolor{purple}{[Yiheng says: #1]}}

\allowdisplaybreaks

\title{Perturbation-based Regret Analysis of Predictive Control in Linear Time Varying Systems
}

\author[1]{Yiheng Lin\thanks{Corresponding author: \textit{yihengl (at) caltech (dot) edu}}$^\dagger$}
\author[2]{Yang Hu$^\dagger$}
\author[1]{Haoyuan Sun$^\dagger$}
\author[1]{Guanya Shi$^\dagger$}
\author[1]{Guannan Qu$^\dagger$}
\author[1]{Adam Wierman}
\affil[1]{California Institute of Technology, Pasadena, CA, USA}
\affil[2]{Tsinghua University, Beijing, China}

\date{June, 2021}

\begin{document}

\maketitle

\renewcommand{\thefootnote}{\fnsymbol{footnote}}
\footnotetext[2]{Yiheng Lin, Yang Hu, Haoyuan Sun, Guanya Shi, and Guannan Qu contributed equally to this work.}
\renewcommand{\thefootnote}{\arabic{footnote}}

\begin{abstract}
  We study predictive control in a setting where the dynamics are time-varying and linear, and the costs are time-varying and well-conditioned. At each time step, the controller receives the exact predictions of costs, dynamics, and disturbances for the future $k$ time steps. We show that when the prediction window $k$ is sufficiently large, predictive control is input-to-state stable and achieves a dynamic regret of $O(\decayvar^k T)$, where $\decayvar < 1$ is a positive constant. This is the first dynamic regret bound on the predictive control of linear time-varying systems. Under more assumptions on the terminal costs, we also show that predictive control obtains the first competitive bound for the control of linear time-varying systems:  $1 + O(\decayvar^k)$. Our results are derived using a novel proof framework based on a perturbation bound that characterizes how a small change to the system parameters impacts the optimal trajectory.

\end{abstract}

\section{Introduction}


We study the problem of predictive control in a linear time-varying (LTV) system, where the dynamics is given by $x_{t+1} = A_t x_t + B_t u_t + w_t$. Here, $x_t$ is the state, $u_t$ is the control action, and $w_t$ is the disturbance or exogenous input. At each time step $t$, the online controller incurs a time-varying state cost $f_t(x_t)$ and control cost $c_t(u_{t-1})$, and then decides its next control action $u_t$. In deciding $u_t$ the controller makes use of predictions of the next $k$ future disturbances, cost functions, and dynamical matrices, and seeks to minimize its total cost on a finite horizon $T$. Our main results bound the dynamic regret and competitive ratio of predictive controllers in this LTV setting.

Recently, a growing literature has sought to design controllers that achieve learning guarantees such as static regret \cite{agarwal2019logarithmic,simchowitz2020naive}, dynamic regret \cite{li2019online,yu2020power}, and competitive ratio \cite{shi2020online}. 
The most relevant line of work concerns predictive control with learning guarantees, which studies how to leverage the prediction window $k$ to reduce the regret and competitive ratio. This line of work has focused on linear time-invariant (LTI) systems \cite{yu2020power, yu2020competitive, zhang2021regret, li2019online}. 
However, linear time-varying (LTV) systems have received increasing attention in recent years due to their importance in a variety of emerging applications, despite the challenges associated with analysis. For example, in the problem of power grid frequency regulation, the dynamics is determined by the proportion of renewable energy in total power generation, which is time-varying~\cite{gonzalez2019powergrid,qu2021stable}. It is also common to use the LTV systems as an approximation of nonlinear dynamics in predictive control and planning \cite{nocedal2006sequential,morgan2014model,shi2020neural,falcone2007linear}.


The current lack of progress toward understanding measures like regret and competitive ratio in LTV settings is due to the need for new techniques to generalize the dynamics from LTI to LTV and the costs from quadratic to well-conditioned functions.  Specifically, the proof approaches used in previous studies on regret and competitive ratio of predictive control in LTI dynamics with quadratic costs, e.g., \cite{yu2020power, yu2020competitive, zhang2021regret}, require explicitly writing down the cost-to-go function, optimal control actions, and algorithm's actions as functions of the system parameters. This is very difficult, if not impossible, for general cost functions that do not have a quadratic form. A promising approach that does not require such explicit characterizations is to derive results via reductions from optimal control to online convex optimization with multi-step memory, e.g.,  \cite{li2019online, shi2020online, goel2018smoothed, agarwal2019online, agarwal2019logarithmic}.
However, such reductions usually do not work well for LTV systems due to the need to represent the problem in control canonical form \cite{li2019online, shi2020online}, or due to limitations on the policy class and comparisons to static benchmarks \cite{agarwal2019online, agarwal2019logarithmic}. 



Perhaps the most prominent approach for controlling LTV systems is Model Predictive Control (MPC), also known as Receding Horizon Control \cite{garcia1989model}. Generally speaking, at each time step, an MPC-style algorithm solves a predictive trajectory for the future $k$ time steps and commit the first control action in this trajectory. MPC-style algorithms are known to work well in practice, even when the dynamics are non-linear and time-varying, e.g., \cite{rosolia2017learning,korda2018linear,allgower2012nonlinear,falcone2007linear}.  
On theoretical side, the asymptotic behaviors of MPC such as stability and convergence have been studied intensively under general assumptions on dynamics and costs \cite{diehl2010lyapunov, angeli2011average, angeli2016theoretical, grune2020economic}. However, non-asymptotic guarantees such as regret and competitive ratio of MPC-style policies have been limited.
Despite recent work providing such guarantees in the context of LTI systems with quadratic costs, e.g., \cite{yu2020power, yu2020competitive, zhang2021regret}, the derivation of regret and competitive ratio results for MPC in LTV systems remains open. 

\textbf{Contributions.}  We provide the first regret and competitive ratio results for a controller in LTV systems with time-varying costs. Specifically, we show that an MPC-style predictive control algorithm (Algorithm \ref{alg:pc}) achieves a dynamic regret that decays exponentially with respect to the length of prediction window $k$ in the LTV system (Theorem \ref{thm:distance-MPC-infty}): $O(\decayvar^k T)$, where the decay rate $\decayvar$ is a positive constant less than $1$. 
This almost matches the exponential lower bound for improvement from predictions in the LTI setting shown in \cite{li2019online} in the sense that, to achieve any target regret level, the required length of prediction $k$ shown by our bound differs from the theoretical lower bound by at most a constant factor. We also show the first competitive bound in LTV systems with time-varying well-conditioned costs (Theorem \ref{thm:new-competitive-ratio}):  $1 + O(\decayvar^k)$, where the decay rate $\decayvar$ is identical with the one in the regret bound.

We develop a novel analysis framework based on a perturbation approach. Specifically, instead of solving for the optimal states and control actions like previous analyses in the LTI setting with quadratic costs, e.g., \cite{yu2020power, zhang2021regret}, we bound how much impact an perturbation to the system parameters can have on the optimal solution. This type of perturbation bound (Theorem \ref{thm:LTV-sensitivity}) can be shown even when the optimal trajectory cannot be written down explicitly, which allows it to be applied in LTV systems with well-conditioned costs. Then, we utilize this perturbation bound to establish results on dynamic regret and the competitive ratio. In addition, we want to emphasize that the perturbation approach we develop is highly modular and extendable. For instance, if a stronger perturbation bound for some specific class of dynamics and/or cost functions can be shown, the dynamic regret of the predictive controller will improve. Similarly, to further generalize the problem setting (e.g., to include additional constraints), one only needs to establish the corresponding perturbation bounds and the regret result will follow.

Another important component of the proof is a novel reduction between LTV control and online optimization.  Connections between online optimization and control have received increasing attention in recent years, e.g., \cite{shi2020online, li2019online, goel2018smoothed, agarwal2019online, agarwal2019logarithmic}. Existing reductions rely on the canonical form, which does not apply to LTV systems, and/or formulations of online optimization with memory of multiple prior time steps, which makes the online problem more challenging. The reduction we present here relies on neither, and is thus a fundamentally different approach to connect control and online optimization. Further, this reduction is not specific to the predictive control algorithm we study, and we expect it to prove useful for other controllers in future work.
A limitation of our reduction framework is that it cannot handle state/control constraints. This limitation is shared by previous works \cite{yu2020power, yu2020competitive, zhang2021regret, li2019online}, and represents a challenging open question in the literature.

\section{Background and Setting}

We consider a finite-horizon discrete-time online control problem with linear time-varying (LTV) dynamics, time-varying costs, and disturbances, namely
\begin{align}\label{equ:online_control_problem}
    \min_{x_{0:T}, u_{0:T-1}} &\sum_{t = 1}^{T} \left(f_{t}(x_t) + c_t(u_{t-1})\right)\nonumber\\*
    \text{ s.t. }&x_t = A_{t-1} x_{t-1} + B_{t-1} u_{t-1} + w_{t - 1}, t = 1, \ldots, T,\\*
    &x_0 = x(0),\nonumber
\end{align}
where $x_t \in \mathbb{R}^n$, $u_t \in \mathbb{R}^m$, and $w_t \in \mathbb{R}^n$ respectively denote the state, the control action, and the disturbance of the system at time steps $t = 1, \ldots, T$, and $x(0) \in \mathbb{R}^n$ is a given initial state. By convention, the hitting cost function $f_t: \mathbb{R}^n \to \mathbb{R}_+$ and control cost function $c_t: \mathbb{R}^m \to \mathbb{R}_+$ are assumed to be time-varying and well-conditioned. Define the tuple $\vartheta_t := (A_{t}, B_{t}, w_{t}, f_{t+1}, c_{t+1})$.

In the classical setting where no predictions are available, after observing state $x_t$ at time step $t$, the algorithm needs to decide the control action $u_t$ before observing $\vartheta_t$, which is an unknown random disturbance input. We use the following event sequence to describe this ordering:
\[x_0, u_0, \vartheta_0, x_1, u_1, \vartheta_1, x_2, \ldots, x_{T-1}, u_{T-1}, \vartheta_{T-1}, x_T.\]
We assume that the algorithm has access to the exact predictions of disturbances, cost functions and dynamical matrices in the future $k$ time steps (which are time-varying); i.e., the event sequence is
\[x_0, \vartheta_0, \vartheta_1, \ldots, \vartheta_{k-1}, u_0, \vartheta_k, u_1, \vartheta_{k+1}, \ldots, u_{T-k-1}, \vartheta_{T-1}, u_{T - k}, u_{T-k+1}, \ldots, u_{T-1}.\]
Here we assume all predictions are \textit{exact}, and leave the case of inexact predictions for future work. This prediction model has been used in previous works like \cite{yu2020power, lin2020online, li2020online, lin2012online}, and is available in many real-world applications such as disturbance estimation in robotics and frequency regulation in power grids. The availability is due to the fact that, in such scenarios as mentioned above, experiments or observations on the dynamics can be conducted repeatedly and consistently,
which makes it feasible to train a good predictor based on the data collected from repeated trials.

\subsection{Assumptions}
\label{sec:assumptions}
As is standard in studies of regret and competitive ratio in linear control problems, we assume the cost functions are well-conditioned.  

\begin{assumption}[Well-conditioned Costs]\label{assump:well-condition}
    The cost functions satisfy the following constraints:
    \begin{enumerate}
        \item $f_t(\cdot)$ is $m_f$-strongly convex for $t = 1, \ldots, T$, and $\ell_f$-strongly smooth for $t = 1, \ldots, T-1$.
        \item $c_t(\cdot)$ is both $m_c$-strongly convex and $\ell_c$-strongly smooth for $t = 1, \ldots, T$.
        \item $f_t(\cdot)$ and $c_t(\cdot)$ are twice continuously differentiable for $t = 1, \ldots, T$.
        \item $f_t(\cdot)$ and $c_t(\cdot)$ are non-negative, and $f_t(0) = c_t(0) = 0$ for $t = 1, \ldots, T$. 
    \end{enumerate}
\end{assumption}
Note that assumptions (1) through (3) are quite common \cite{li2019online, li2020online, goel2018smoothed, goel2019beyond, shi2020online}. Assumption (4) is less common, but can be satisfied via re-parameterization without loss of generality. Specifically, when the minimizers of state cost $f_t$ and control cost $c_t$ are nonzero, we perform the transformation
\begin{gather*}
    x_t' \leftarrow x_t - \argmin_x f_t(x),~ u_t' \leftarrow u_t - \argmin_u c_{t+1}(u),\\
    w_t' \leftarrow w_t + A_t \argmin_x f_t(x) + B_t \argmin_u c_{t+1}(u).
\end{gather*}

Additionally, we need to assume the dynamics are \textit{controllable}.  It is crucial that the dynamical system can be steered from an arbitrary initial state to an arbitrary final state via a finite sequence of admissible control actions. For linear time-invariant (LTI) systems, the full-rankness of the \textit{controllability matrix} completely characterizes the reachability of the state space, which is generally used as a standard assumption for analysis \cite{zhang2021regret, mania2019certainty, aastrom2010feedback}. This can be generalized to parallel  assumptions for LTV systems as follows. We begin with a definition.

\begin{definition}\label{def:controllability}
    For a dynamical system with linear time-varying dynamics
    $x_t = A_{t-1} x_{t-1} + B_{t-1} u_{t-1} + w_{t - 1}, t = 1, \ldots, T,$
    the transition matrix $\Phi(t_2, t_1) \in \mathbb{R}^{n \times n}$ (from time step $t_1$ to $t_2$) is defined as
    \[\Phi(t_2, t_1) := \begin{cases}
        A_{t_2 - 1} A_{t_2 - 2} \cdots A_{t_1} & \text{ if }t_2 > t_1\\
        I & \text{ if } t_2 \leq t_1
    \end{cases},\]
    and the controllability matrix $M(t, p) \in \mathbb{R}^{n \times (mp)}$ is defined as
    \[M(t, p) := \left[\Phi(t+p, t+1) B_t, \Phi(t+p, t+2) B_{t+1}, \ldots, \Phi(t+p, t+p) B_{t+p}\right].\]
    The dynamical system is called controllable if there exists a constant $d \in \mathbb{Z}_+$, such that the controllability matrix $M(t, d)$ is of full row rank for any $t = 1, \ldots, T-d$. The smallest constant $d$ with such property is called the controllability index of the system.
\end{definition}

Given the above definition, we can state the key assumption necessary for the analysis of LTV systems. We use a slightly stronger assumption than being merely controllable, which we refer to as $(d, \sigma)$-uniform controllability.  It is a natural generalization of its counterpart for LTI systems (see Assumption 2 in \cite{mania2019certainty}, where $(d, \sigma)$ is instead named as $(\ell, \nu)$).

\begin{assumption}\label{assump:LTV}
    There exists positive constants $a$, $b$, and $b'$, such that 
    \[\norm{A_t} \leq a,~ \norm{B_t} \leq b,~ \text{and}~\Vert B_t^\dagger\Vert \leq b'\]
    hold for all time steps $t = 0, \ldots, T-1$, where $B_t^\dagger$ denotes the Moore–Penrose inverse of matrix $B_t$. Furthermore, there exists a positive constant $\sigma$ such that
    \[\sigma_{\min} \left( M(t, d) \right) \geq \sigma\]
    holds for all time steps $t = 0, \ldots, T - d$, where $d$ denotes the controllability index.
\end{assumption}

Note that Assumption \ref{assump:LTV} implies $\sigma_{\min}(M(t, p)) \geq \sigma$ for all $p \geq d$ because appending more columns to a matrix with full row rank will not reduce its minimum singular value. 


The LTV setting we consider is more general than the settings which existing results on regret and competitive ratio have assumed \cite{yu2020power, zhang2021regret, li2019online, agarwal2019logarithmic}. We highlight the implications of this general setting for enabling applications in the following examples.

\begin{example}[Trajectory tracking in LTV systems with well-conditioned costs]
    Consider a trajectory tracking problem with LTV dynamics and well-conditioned costs, which generalizes the standard linear quadratic tracking problem in \cite{brian2007optcontrol,yu2020power} with LTI dynamics and quadratic costs. We adopt LTV dynamics $x_{t+1} = A_t x_t + B_t u_t + w_t$ and general well-conditioned cost functions $f_t(\cdot), c_t(\cdot)$ (see Assumption \ref{assump:well-condition}). With the desired trajectory $d_{1:T}$, we consider a new state $\tilde{x}_t := x_t - d_t$ and a new disturbance $\tilde{w}_t := w_t + A_t d_t - d_{t+1}$. Thus, using the new state and disturbance, the problem naturally fits into our problem setting with $k$ future predictions of $(A_t, B_t, w_t, f_t, c_t, d_{t+1})$. Note that predictive control with LTV dynamics is practical in nonlinear systems \cite{falcone2008linear,falcone2007linear} because the nonlinearity could be well approximated by LTV models \cite{falcone2008linear}.
\end{example}

\begin{example}[Power grid frequency regulation] 
    Consider the frequency regulation problem in \cite{gonzalez2019powergrid}, where state $x = [\theta^{\top}, \omega^{\top}]^{\top}$ represent the status of a power plant, and power generation $p_{\text{in}} \in \bR^n$ is the control action. The continuous-time dynamics is given by
    \[\underbrace{\begin{bmatrix}
        \dot{\theta} \\ \dot{\omega}
      \end{bmatrix}}_{\dot{x}(t)}
      = \underbrace{\begin{bmatrix}
        0 & I \\
        -M(t)^{-1}L & -M(t)^{-1}D
      \end{bmatrix}}_{\hat{A}(t)} 
      \underbrace{\begin{bmatrix}
        \theta \\ \omega
      \end{bmatrix}}_{x(t)}
      + \underbrace{\begin{bmatrix}
        0 \\ M(t)^{-1}
      \end{bmatrix}}_{\hat{B}(t)}
      \underbrace{p_{\text{in}}}_{u(t)}. \]
    Here $M(t)$ denotes the rotational inertia matrix, which is time-varying and is determined by the proportion of renewable power in total power generation at time $t$, and can be accurately predicted in a certain time horizon \cite{santhosh2020windforecast,nageem2017solarforecast}; $L$ and $D$ are known system parameters.
    Using standard discretization techniques, we can formulate a discrete-time linear time-varying system $x_{t+1} = A_t x_t + B_t u_t + w_t$, where $A_t$ and $B_t$ are determined by $\hat{A}(t)$ and $\hat{B}(t)$. The cost functions are quadratic costs which penalizes frequency deviation \cite{gonzalez2019powergrid}. 
    This setting fits into our predictive control algorithm, since the controllers have accurate predictions of $A_t$ and $B_t$ in the near future due to predictablity of $M(t)$.
\end{example}

\subsection{Predictive Control}
We study a classical predictive control (PC) algorithm inspired by model predictive control, which solves the optimization problem of $k$ future time steps (where $k$ is called the \textit{prediction window}). specifically, the algorithm receives the dynamics and disturbances of the next $k$ time steps, calculates the optimal solution, and then applies the first control action of the optimal solution. The PC algorithm with prediction window $k$ is denoted as $PC_{k}$.

More formally, At time step $t < T - k$, $PC_{k}$ solves the optimization problem $\tilde{\psi}_t^k(x_t, w_{t:t+k-1}; F)$. Since we need to consider horizon lengths other than $k$, for arbitrary $\horizonvar \geq 1$ and time step $t$, we define the optimization problem $\tilde{\psi}_{t}^\horizonvar (x, \zeta; F)$ as
\begin{align}\label{opt:without-terminal}
    \tilde{\psi}_{t}^\horizonvar (x, \zeta; F) := \argmin_{y_{0:\horizonvar}, v_{0:\horizonvar-1}} &\sum_{\tau = 1}^{\horizonvar} f_{t + \tau}(y_\tau) + \sum_{\tau = 1}^{\horizonvar} c_{t + \tau}\left(v_{\tau-1}\right) + F(y_{k})\nonumber\\*
    \text{ s.t. }&y_\tau = A_{t+\tau-1} y_{\tau - 1} + B_{t+\tau-1} v_{\tau-1} + \zeta_{\tau - 1}, \tau = 1, \ldots, \horizonvar,\\*
    &y_0 = x,\nonumber
\end{align}
where $x \in \mathbb{R}^n$ is the initial state, $\zeta \in \left(\mathbb{R}^n\right)^{\horizonvar}$ (indexed by $0, \ldots, \horizonvar-1$) is a sequence of disturbances, and $F: \mathbb{R}^n \to \mathbb{R}$ is a standard terminal cost function regularizing the final state. 
We will specify some additional requirements on $F$ later. For example, to derive the competitive ratio result, we require $F$ to be the indicator function of the origin\footnote{When we say the indicator function of the origin in this paper, we refer to a function $g: \mathbb{R}^n \to \mathbb{R}\cup \{+\infty, -\infty\}$ that is defined as $g(0) = 0$ and $g(x) = +\infty$ for any $x \not = 0$}.
For each time step $\tau = 1, \ldots, k$, $y_\tau \in \mathbb{R}^n$ is the predictive state, and $v_\tau \in \mathbb{R}^m$ is the predictive control action. To make the algorithm well-defined, at time step $t = T - k$, $PC_k$ can finish the rest of the trajectory optimally by committing $u_{T-k:T-1} = \tilde{\psi}_{T-k}^k(x_{T-k}, w_{T-k:T-1}; 0)$. We will add a subscript $y_i/v_i$ to the optimal solution vector to denote one of its entries. The pseudo-code of predictive control is given in Algorithm \ref{alg:pc}. 


\begin{algorithm}[t]
\caption{Predictive Control ($PC_{k}$)}\label{alg:pc}
\begin{algorithmic}[1]
\For{$t = 0, 1, \ldots, T-k-1$}
    \State Observe current state $x_t$ and receive predictions $\vartheta_{t:t+k-1}$.
    \State Solve and commit control actions $u_t := \tilde{\psi}_t^k(x_t, w_{t:t+k-1}; F)_{v_0}$.
\EndFor
\State At time step $t = T-k$, observe current state $x_t$ and receive predictions $\vartheta_{t:T-1}$.
\State Solve and commit control actions $u_{t:T-1} := \tilde{\psi}_t^{k}(x_t, w_{t:T-1}; 0)_{v_{0:k-1}}$.
\end{algorithmic}
\end{algorithm}

It is also important in our framework to study the behavior of predictive control under some fixed terminal point.
So, for prediction length $\horizonvar \geq 1$ and time step $t$, we define an auxiliary optimization problem with a strict terminal constraint $y_{\horizonvar} = z$ as follows:
\begin{align}\label{opt:with-terminal}
    \psi_{t}^\horizonvar(x, \zeta, z) := \argmin_{y_{0:\horizonvar}, v_{0:\horizonvar-1}} &\sum_{\tau = 1}^\horizonvar f_{t + \tau}(y_\tau) + \sum_{\tau = 1}^\horizonvar c_{t + \tau}\left(v_{\tau-1}\right)\nonumber\\*
    \text{ s.t. }&y_\tau = A_{t+\tau-1} y_{\tau - 1} + B_{t+\tau-1} v_{\tau-1} + \zeta_{\tau - 1}, \tau = 1, \ldots, \horizonvar,\\*
    & y_0 = x, y_\horizonvar = z,\nonumber
\end{align}
where the optimal value is denoted by 
$\iota_{t}^{\horizonvar}(x, \zeta, z)$. We define this auxiliary optimization problem besides $\tilde{\psi}_{t}^\horizonvar (x, \zeta; F)$ because we need to fix both the initial state and the terminal state, for example, when expressing a sub-trajectory of the offline optimal trajectory as the solution of an optimization problem. $\psi$ also allows us to study the impact of the perturbation at the terminal state on the optimal trajectory directly, which will be useful in the proof of dynamic regret (Theorem \ref{thm:distance-MPC-infty}) and competitive ratio (Theorem \ref{thm:new-competitive-ratio}). In addition, when the terminal cost $F$ is the indicator function of the origin, we can relate $\tilde{\psi}$ to $\psi$ as
$\tilde{\psi}_{t}^\horizonvar (x, \zeta; F) = \psi_{t}^{\horizonvar}(x, \zeta, 0).$
This conversion will be useful in the proof of the competitive ratio result (Theorem \ref{thm:new-competitive-ratio}).


Throughout the paper, we use $\{(x_t, u_t)\}_{t=1}^T$ to denote the trajectory of predictive control, and use $\{(x_t^*, u_t^*)\}_{t=1}^T$ to denote the offline optimal trajectory (i.e., the optimal solution of \eqref{equ:online_control_problem}). We also use several standard definitions and notations in linear algebra and optimization, which we detail in Appendix \ref{sec:standard-notations} for clarity. In particular, we use vector 2-norms and induced matrix 2-norms throughout this paper unless otherwise specified.

\section{A Perturbation Approach}
\label{sec:perturbation}

In order to study the regret and competitive ratio of controllers in LTV systems, we develop a new analysis based on a perturbation approach, which we introduce in this section.  This approach is based on developing  bounds on how much the solutions to \eqref{opt:without-terminal} and \eqref{opt:with-terminal} change with respect to perturbations to the initial/terminal states and the disturbance sequence. Our perturbation bounds are related to the concept of \textit{incremental stability} defined in \cite{tran2016incremental}, but not exactly the same because we consider the optimal trajectory in a finite horizon whereas the incremental stability focuses on asymptotic behavior over an infinite horizon. Simply stated, the key to our approach is to derive the perturbation bound in Theorem \ref{thm:LTV-sensitivity}, which states that if the target variable we are concerned with is the $h$-th predictive state/control input, while the perturbation occurs at the $\tau$-th time step, then the impact on the target variable is be exponentially small with respect to the time difference $\abs{h - \tau}$. 
This result can also be viewed as a special exponential decay property possessed by the derivative of the optimal solution with respect to the system parameters in an optimal control problem. Differentiating optimization problems has been considered in a line of previous works (e.g. \cite{amos2018differentiable, amos2017optnet, agrawal2019differentiable}), but they focus more on how to compute the derivatives rather than the exponential decay property we discuss here.



Proving such a result directly is challenging because of the complexity of the LTV dynamical constraints in \eqref{opt:without-terminal} and \eqref{opt:with-terminal}. Thus, we develop a novel reduction from LTV systems to fully-actuated systems, i.e., systems where the controller can steer the system to any state in the whole space $\mathbb{R}^n$ freely at every time step. This special case is a form of online optimization called \textit{smoothed online convex optimization} (SOCO), and has received considerable attention recently, e.g., \cite{goel2019beyond, lin2020online, argue2020dimension}. 
We exploit the controllability of the dynamics to analyze the LTV system in chunks of $d$ time steps.
A sequence of $d$ time steps combined together can be thought as a fully-actuated system and thus we can formulate a SOCO problem, which is $(1/d)$-times as long as the original LTV system. 
In this section, we first show the perturbation bound for SOCO in Section \ref{sec:perturbation-soco}, and then we leverage our reduction to derive a result for general LTV systems in Section \ref{sec:perturbation-ltv}.

\subsection{Smoothed Online Convex Optimization}
\label{sec:perturbation-soco}

The classic setting of SOCO is an online game played by an agent against an adversary: at each time step $t$, the adversary reveals a hitting cost function $\hat{f}_t$, a switching cost function $\hat{c}_t$, and a disturbance (or exogenous input) $\hat{w}_t$. The agent picks a decision point $\hat{x}_t \in \mathbb{R}^n$, and incurs a stage cost of $\hat{f}_t(\hat{x}_t) + \hat{c}_t(\hat{x}_t, \hat{x}_{t-1}, \hat{w}_{t-1})$. The agent seeks to minimize the total cost it incurs throughout the game. The offline optimal cost is defined as the minimum cost if the agent has full knowledge of the costs and disturbances at the start of the game. Instead of analyzing the performance of an online algorithm directly, our focus is on studying how the perturbations of the system parameters (initial state, terminal state, and disturbances) impact the offline optimal solution. These results are critical for deriving the guarantees for predictive control in the online setting in Section \ref{sec:guarantees}.

To begin, observe that when the initial state $\hat{x}_0$, terminal state $\hat{x}_\horizonvar$, and the disturbances $\hat{w}$ are given, the optimal $\horizonvar$-step trajectory of SOCO can be obtained from the unconstrained optimization problem
\begin{equation}\label{equ:SOCO-opt}
    \hat{\psi}(\hat{x}_0, \hat{w}, \hat{x}_\horizonvar) := \argmin_{\hat{x}_{1:\horizonvar-1}} \sum_{\tau=1}^{\horizonvar-1} \hat{f}_\tau(\hat{x}_\tau) + \sum_{\tau=1}^{\horizonvar} \hat{c}_\tau(\hat{x}_\tau, \hat{x}_{\tau - 1}, \hat{w}_{\tau-1}),
\end{equation}
where the objective is a convex function of the decision variables $\hat{x}_{1:\horizonvar-1}$. Since \eqref{equ:SOCO-opt} is an unconstrained optimization problem, the gradient of its objective equals zero at $\hat{\psi}(\hat{x}_0, \hat{w}, \hat{x}_\horizonvar)$. Using this, we can further show that the directional derivative of $\hat{\psi}(\hat{x}_0, \hat{w}, \hat{x}_\horizonvar)$ along some direction $e$, denoted by $\chi$, satisfies the linear equation $M \chi = \delta$, where symmetric matrix $M$ is the Hessian of the objective and vector $\delta$ is determined by the direction $e$.
A special structure of the objective of \eqref{equ:SOCO-opt} is that the correlations only occur in two consecutive time steps. This implies that its Hessian $M$ is block tri-diagonal. Such tri-diagonal structure of $M$ has been noted by previous work, e.g. \cite{amos2018differentiable}, and have been leveraged to solve the linear equation $M \chi = \delta$ quickly. In contrast, we focus on the exponential decay phenomena $M^{-1}$ exhibits, i.e., the magnitudes of entries decay exponentially with respect to their distances to the main diagonal \cite{demko1984decay}. Bounding each entry of $\chi = M^{-1} \delta$ separately gives us the following perturbation bound. We state this result formally in Theorem \ref{thm:SOCO-sensitivity}, and its proof can be found in Appendix \ref{Appendix:thm:SOCO-sensitivity}. 

\begin{theorem}\label{thm:SOCO-sensitivity}
Given a tuple $(\hat{x}_0, \hat{w}, \hat{x}_\horizonvar)$ that contains the initial state, the disturbances, and the terminal state in this order, we consider the optimal solution of the SOCO problem
\[\hat{\psi}(\hat{x}_0, \hat{w}, \hat{x}_\horizonvar) := \argmin_{\hat{x}_{1:\horizonvar-1}} \sum_{\tau=1}^{\horizonvar-1} \hat{f}_\tau(\hat{x}_\tau) + \sum_{\tau=1}^{\horizonvar} \hat{c}_\tau(\hat{x}_\tau, \hat{x}_{\tau - 1}, \hat{w}_{\tau-1})\]
indexed by $1, \ldots, \horizonvar-1$. Assume $\hat{f}_\tau: \mathbb{R}^n \to \mathbb{R}$ is $\mu$-strongly convex, $\hat{c}_\tau: \mathbb{R}^n \times \mathbb{R}^n \times \mathbb{R}^r \to \mathbb{R}$ is convex and $\ell$-strongly smooth, and both are twice continuously differentiable for $\tau = 1, \ldots, \horizonvar$, then
\begin{equation*}
    \norm{\hat{\psi}(\hat{x}_0, \hat{w}, \hat{x}_\horizonvar)_h - \hat{\psi}(\hat{x}_0', \hat{w}', \hat{x}_\horizonvar')_h} \leq C_0\left(\decayvar_0^{h-1}\norm{\hat{x}_0 - \hat{x}_0'} + \sum_{\tau=0}^{\horizonvar-1} \decayvar_0^{\abs{h - \tau}-1}\norm{\hat{w}_\tau - \hat{w}_\tau'} + \decayvar_0^{\horizonvar-h-1}\norm{\hat{x}_\horizonvar - \hat{x}_\horizonvar'} \right)
\end{equation*}
for all $1 \leq h \leq \horizonvar-1$, where $C_0 = {(2\ell)}/{\mu}$ and $\decayvar_0 = 1 - 2 \cdot \left(\sqrt{1 + (2\ell/\mu)} + 1\right)^{-1}$.
\end{theorem}

As a remark, we do not require the hitting cost $\hat{f}_\tau$ to be strongly smooth, or the switching cost $\hat{c}_\tau$ to be strongly convex in Theorem \ref{thm:SOCO-sensitivity}. This makes the assumptions on the SOCO costs $\hat{f}_\tau, \hat{c}_\tau$ weaker than the assumptions on the LTV costs $f_\tau, c_\tau$ defined in \eqref{equ:online_control_problem}.

\subsection{Linear Time-Varying System}
\label{sec:perturbation-ltv}

We now build upon the SOCO perturbation result to derive a perturbation result for LTV systems.  In particular, we show an exponentially-decaying perturbation bound for our LTV system by reducing it to SOCO and apply Theorem \ref{thm:SOCO-sensitivity}. As we have discussed, LTV systems are more difficult than SOCO because the dynamics prevent the online agent from picking the next state $x_{t+1}$ freely at a given state $x_{t}$. 
We overcome this obstacle by redefining the decision points as illustrated in Figure~\ref{fig:reduction}. Specifically, given state $x_t$ at time step $t$ as the last decision point, we then ask the online agent to decide state $x_{t+d}$ at time step $(t+d)$ rather than $x_{t+1}$ at time step $(t+1)$.

Since $d$ is the controllability index, $x_{t+d}$ can be picked freely from the whole space $\mathbb{R}^n$ regardless of $x_t$. 
We also utilize the 
\textit{principle of optimality}, e.g. if $y_{0:k}, v_{0:k-1}$ is the optimal solution to $\psi^k_t(x, \xi, z)$, then $y_{i:j}, v_{i:j-1}$ is the optimal solution to $\psi^{j-i}_{t + i}(y_i, \xi_{i:j-1}, y_j)$ for any $0 \le i < j \le k$.
Therefore, the trajectory between time $t$ and $(t+d)$ can be recovered by solving $\psi_t^d(x_t, w_{t:t+d-1}, x_{t+d})$.
So we are able to formulate a valid SOCO problem on the sequence of time steps $t, t+d, t+2d, \dots$.

Naturally, the hitting cost at time step $(t+d)$ remains the same, while the switching cost becomes $\xi_t^{d}(x_t, w_{t:t+d-1}, x_{t+d})$, where the function $\xi_t^\horizonvar$ is defined as
\begin{equation}\label{equ:switching-cost}
    \xi_{t}^\horizonvar(x, \zeta, z) := \iota_t^\horizonvar(x, \zeta, z) - f_{t+\horizonvar}(z).
\end{equation}
An illustration of the reduction can be found in Figure \ref{fig:reduction}. Unlike the switching costs in \cite{goel2019beyond, goel2018smoothed,chen2018smoothed, argue2020dimension} which are explicitly defined as the $\ell_2$-distance or squared $\ell_2$-distance, the switching cost $\xi_t^p$ here is defined implicitly as the optimal value of an optimization problem. Lemma \ref{lemma:implicit-switching-cost} shows that the switching cost defined in \eqref{equ:switching-cost} satisfies the requirements of Theorem \ref{thm:SOCO-sensitivity}, which allows us to obtain the desired perturbation bound.

\vspace{-1em}
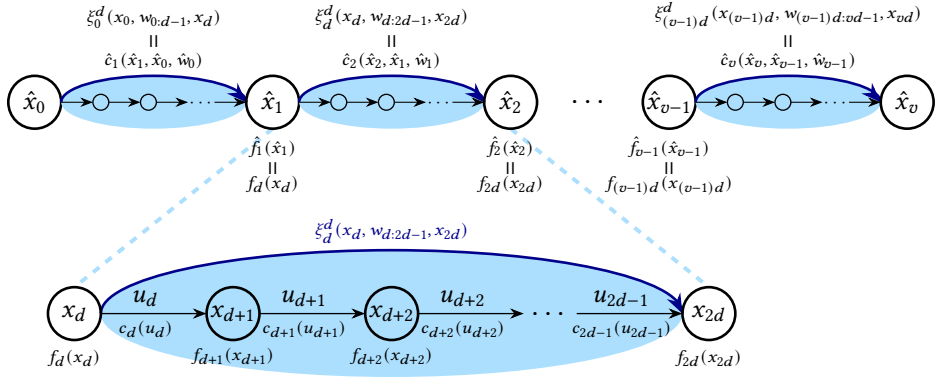
\begin{figure}[htbp]
    \centering
    \begin{tikzpicture}
    \draw[draw=figblue!50!white, line width=1.5pt, dashed] (90pt, 70pt) -- (15pt, 10pt);
    \draw[draw=figblue!50!white, line width=1.5pt, dashed] (180pt, 70pt) -- (255pt, 10pt);

    \begin{scope}[xshift=15pt]
        \fill[figblue!50!white] (120pt, 0pt) circle (110pt and 24pt);
        \draw[DarkBlue, line width=1pt] (10pt, 0pt) arc (180:0:110pt and 24pt);
        \draw[DarkBlue, line width=1pt, ->, >={Stealth}] (228pt, 4pt) -- (230pt, 0pt);
        \node[above=-1pt] at (120pt, 24pt) {\tiny \color{DarkBlue} $\xi_{d}^{d}(x_d, w_{d:2d-1}, x_{2d})$};
        
        \draw[draw=black, line width=1pt] (0pt, 0pt) circle (10pt); 
        \node at (0pt, 0pt) {\small $x_d$};
        \node[below] at (0pt, -10pt) {\tiny $f_d(x_d)$};
        
        \draw[draw=black, line width=0.5pt, ->, >={Stealth}] (10pt, 0pt) -- (50pt, 0pt);
        \node[above=-1pt] at (27pt, 0pt) {\small $u_d$};
        \node[below=-1pt] at (27pt, 0pt) {\tiny \color{black} $c_d(u_d)$};
        
        \draw[draw=black, line width=1pt] (60pt, 0pt) circle (10pt); 
        \node at (60pt, 0pt) {\small $x_{d+1}$};
        \node[below=-1pt] at (60pt, -10pt) {\tiny \color{black} $f_{d+1}(x_{d+1})$};
        
        \draw[draw=black, line width=0.5pt, ->, >={Stealth}] (70pt, 0pt) -- (110pt, 0pt);
        \node[above=-1pt] at (87pt, 0pt) {\small $u_{d+1}$};
        \node[below=-1pt] at (87pt, 0pt) {\tiny \color{black} $c_{d+1}(u_{d+1})$};
        
        \draw[draw=black, line width=1pt] (120pt, 0pt) circle (10pt); 
        \node at (120pt, 0pt) {\small $x_{d+2}$};
        \node[below=-1pt] at (120pt, -10pt) {\tiny \color{black} $f_{d+2}(x_{d+2})$};
        
        \draw[draw=black, line width=0.5pt, ->, >={Stealth}] (130pt, 0pt) -- (170pt, 0pt);
        \node[above=-1pt] at (147pt, 0pt) {\small $u_{d+2}$};
        \node[below=-1pt] at (147pt, 0pt) {\tiny \color{black} $c_{d+2}(u_{d+2})$};
        
        \node at (180pt, -1pt) {$\cdots$};
        
        \draw[draw=black, line width=0.5pt, ->, >={Stealth}] (190pt, 0pt) -- (230pt, 0pt);
        \node[above=-1pt] at (207pt, 0pt) {\small $u_{2d-1}$};
        \node[below=-1pt] at (207pt, 0pt) {\tiny \color{black} $c_{2d-1}(u_{2d-1})$};
        
        \draw[draw=black, line width=1pt] (240pt, 0pt) circle (10pt); 
        \node at (240pt, 0pt) {\small $x_{2d}$};
        \node[below] at (240pt, -10pt) {\tiny $f_{2d}(x_{2d})$};
    \end{scope}
    
    \begin{scope}[yshift=80pt]
        \begin{scope}[xshift=0pt]
            \draw[line width=1pt] (0pt, 0pt) circle (10pt); 
            \node at (0pt, 0pt) {\small $\hat{x}_0$};
            
            \fill[figblue!50!white] (45pt, 0pt) circle (35pt and 10pt);
            \draw[DarkBlue, line width=1pt] (10pt, 0pt) arc (180:0:35pt and 10pt);
            \draw[DarkBlue, line width=1pt, ->, >={Stealth}] (77pt, 3pt) -- (80pt, 0pt);
            \node[above=-1pt] at (45pt, 10pt) {\tiny $\hat{c}_1(\hat{x}_1, \hat{x}_0, \hat{w}_0)$};
            \node[above=7pt] at (45pt, 10pt) {\small $\shortparallel$};
            \node[above=14pt] at (45pt, 10pt) {\tiny $\xi_{0}^{d}(x_0,w_{0:d-1},x_{d})$};
            
            \draw[draw=black, line width=0.4pt, ->, >={Stealth}] (10pt, 0pt) -- (22pt, 0pt);
            \draw[draw=black, line width=0.5pt] (25pt, 0pt) circle (3pt);
            \draw[draw=black, line width=0.4pt, ->, >={Stealth}] (28pt, 0pt) -- (40pt, 0pt);
            \draw[draw=black, line width=0.5pt] (43pt, 0pt) circle (3pt);
            \draw[draw=black, line width=0.4pt, ->, >={Stealth}] (46pt, 0pt) -- (58pt, 0pt);
            \node at (63pt, -1pt) {\tiny $\cdots$};
            \draw[draw=black, line width=0.4pt, ->, >={Stealth}] (68pt, 0pt) -- (80pt, 0pt);
        \end{scope}
        
        \begin{scope}[xshift=90pt]
            \draw[line width=1pt] (0pt, 0pt) circle (10pt); 
            \node at (0pt, 0pt) {\small $\hat{x}_1$};
            \node[below=-1pt] at (0pt, -10pt) {\tiny $\hat{f}_1(\hat{x}_1)$};
            \node[below=9pt] at (0pt, -10pt) {\small $\shortparallel$};
            \node[below=14pt] at (0pt, -10pt) {\tiny $f_{d}(x_{d})$};
            
            \fill[figblue!50!white] (45pt, 0pt) circle (35pt and 10pt);
            \draw[DarkBlue, line width=1pt] (10pt, 0pt) arc (180:0:35pt and 10pt);
            \draw[DarkBlue, line width=1pt, ->, >={Stealth}] (77pt, 3pt) -- (80pt, 0pt);
            \node[above=-1pt] at (45pt, 10pt) {\tiny $\hat{c}_2(\hat{x}_2, \hat{x}_1, \hat{w}_1)$};
            \node[above=7pt] at (45pt, 10pt) {\small $\shortparallel$};
            \node[above=14pt] at (45pt, 10pt) {\tiny $\xi_{d}^{d}(x_{d},w_{d:2d-1},x_{2d})$};
            
            \draw[draw=black, line width=0.4pt, ->, >={Stealth}] (10pt, 0pt) -- (22pt, 0pt);
            \draw[draw=black, line width=0.5pt] (25pt, 0pt) circle (3pt);
            \draw[draw=black, line width=0.4pt, ->, >={Stealth}] (28pt, 0pt) -- (40pt, 0pt);
            \draw[draw=black, line width=0.5pt] (43pt, 0pt) circle (3pt);
            \draw[draw=black, line width=0.4pt, ->, >={Stealth}] (46pt, 0pt) -- (58pt, 0pt);
            \node at (63pt, -1pt) {\tiny $\cdots$};
            \draw[draw=black, line width=0.4pt, ->, >={Stealth}] (68pt, 0pt) -- (80pt, 0pt);
        \end{scope}
        
        \draw[line width=1pt] (180pt, 0pt) circle (10pt); 
        \node at (180pt, 0pt) {\small $\hat{x}_2$};
        \node[below=-1pt] at (180pt, -10pt) {\tiny $\hat{f}_2(\hat{x}_2)$};
        \node[below=9pt] at (180pt, -10pt) {\small $\shortparallel$};
        \node[below=14pt] at (180pt, -10pt) {\tiny $f_{2d}(x_{2d})$};
        
        \node at (210pt, -1pt) {$\cdots$};
        
        \draw[line width=1pt] (240pt, 0pt) circle (10pt); 
        \node at (240pt, 0pt) {\small $\hat{x}_{v-1}$};
        \node[below=-1pt] at (240pt, -10pt) {\tiny $\hat{f}_{v-1}(\hat{x}_{v-1})$};
        \node[below=9pt] at (240pt, -10pt) {\small $\shortparallel$};
        \node[below=14pt] at (240pt, -10pt) {\tiny $f_{(v-1)d}(x_{(v-1)d})$};
        
        \begin{scope}[xshift=250pt]
            \fill[figblue!50!white] (35pt, 0pt) circle (35pt and 10pt);
            \draw[DarkBlue, line width=1pt] (0pt, 0pt) arc (180:0:35pt and 10pt);
            \draw[DarkBlue, line width=1pt, ->, >={Stealth}] (67pt, 3pt) -- (70pt, 0pt);
            \node[above=-1pt] at (35pt, 10pt) {\tiny $\hat{c}_v(\hat{x}_v, \hat{x}_{v-1}, \hat{w}_{v-1})$};
            \node[above=7pt] at (35pt, 10pt) {\small $\shortparallel$};
            \node[above=14pt] at (35pt, 10pt) {\tiny $\xi_{(v-1)d}^{d}(x_{(v-1)d},w_{(v-1)d:vd-1},x_{vd})$};
            
            \draw[draw=black, line width=0.4pt, ->, >={Stealth}] (0pt, 0pt) -- (12pt, 0pt);
            \draw[draw=black, line width=0.5pt] (15pt, 0pt) circle (3pt);
            \draw[draw=black, line width=0.4pt, ->, >={Stealth}] (18pt, 0pt) -- (30pt, 0pt);
            \draw[draw=black, line width=0.5pt] (33pt, 0pt) circle (3pt);
            \draw[draw=black, line width=0.4pt, ->, >={Stealth}] (36pt, 0pt) -- (48pt, 0pt);
            \node at (53pt, -1pt) {\tiny $\cdots$};
            \draw[draw=black, line width=0.4pt, ->, >={Stealth}] (58pt, 0pt) -- (70pt, 0pt);
            
            \draw[line width=1pt] (80pt, 0pt) circle (10pt); 
            \node at (80pt, 0pt) {\small $\hat{x}_v$};
        \end{scope}
    \end{scope}
\end{tikzpicture}
    \caption{
    Illustration of the reduction from LTV to SOCO. Here we consider a simple example where $t=0$ and $\horizonvar = vd$. At time step $0$, the agent cannot steer the system to an arbitrary target state at the next time step due to dynamical constraints. However, given $(d, \sigma)$-uniform controllability, the controller is able to enforce an arbitrary target state after $d$ time steps, which prompts the transformation to a SOCO problem with a decision point in every $d$ time steps.
    }
    \label{fig:reduction}
\end{figure}

\begin{lemma}\label{lemma:implicit-switching-cost}
Under Assumption \ref{assump:well-condition} and \ref{assump:LTV}, for integer $\horizonvar \geq d$, we have
\begin{enumerate}
    \item $\psi_{t}^\horizonvar (x, \zeta, z)$ is $L_1(\horizonvar)$-Lipschitz in $(x, \zeta, z)$;
    \item $\xi_{t}^\horizonvar(x, \zeta, z)$ is convex and $L_2(\horizonvar)$-strongly smooth in $(x, \zeta, z)$.
\end{enumerate}
Here $L_1(\horizonvar) = C(\horizonvar)\left(1 + \ell \cdot C(\horizonvar)/m_c\right), L_2(\horizonvar) = \ell \cdot C(\horizonvar)^2 + {\ell^2\cdot C(\horizonvar)^4}/{m_c},$
where $\ell = \max(\ell_f, \ell_c)$,
\begin{equation*}
    C(\horizonvar) = \begin{cases}
    O(a^{3\horizonvar}) & \text{ if } a > 1;\\
    O(\horizonvar^2) & \text{ if } a = 1;\\
    O(1) & \text{ if } a < 1.
    \end{cases}
\end{equation*}
\end{lemma}

In Lemma \ref{lemma:implicit-switching-cost}, we use $O(\cdot)$ to hide quantities $a, b$, and $1/\sigma$; the precise expression of $C(\horizonvar)$ and the proof of Lemma \ref{lemma:implicit-switching-cost} can be found in Appendix \ref{Appendix:lemma:implicit-switching-cost}. Using the reduction from LTV to SOCO, we obtain a perturbation bound for the LTV systems \eqref{opt:without-terminal} and \eqref{opt:with-terminal} in Theorem \ref{thm:LTV-sensitivity}, the proof of which is deferred to Appendix \ref{Appendix:thm:LTV-sensitivity}.


\begin{theorem}\label{thm:LTV-sensitivity}
Consider the optimization problem defined in \eqref{opt:without-terminal} and \eqref{opt:with-terminal} and with a horizon length $\horizonvar \ge d$. Suppose the terminal cost $F$ is either the indicator function of the origin, or a non-negative convex function that is twice continuously differentiable and satisfies $F(0) = 0$.
Under Assumptions \ref{assump:well-condition} and \ref{assump:LTV}, given any $(x, \zeta, z)$ and $(x', \zeta', z')$,
\begin{align*}
    \norm{\tilde{\psi}_{t}^\horizonvar(x, \zeta; F)_{y_h} - \tilde{\psi}_{t}^\horizonvar (x', \zeta'; F)_{y_h}} &\leq C \left(\decayvar^h \norm{x - x'} + \sum_{\tau = 0}^{\horizonvar-1}\decayvar^{\abs{h - \tau}} \norm{\zeta_\tau - \zeta_\tau'}\right)\\
    \norm{\psi_{t}^\horizonvar(x, \zeta, z)_{y_h} - \psi_{t}^\horizonvar (x', \zeta', z')_{y_h}} &\leq C \left(\decayvar^h \norm{x - x'} + \sum_{\tau = 0}^{\horizonvar-1}\decayvar^{\abs{h - \tau}} \norm{\zeta_\tau - \zeta_\tau'} + \decayvar^{\horizonvar-h} \norm{z - z'}\right)
\end{align*}
hold for all time steps $t$. Here we define $L_0 = \max_{d \leq \horizonvar \leq 2d-1}L_2(\horizonvar)$, and the constants are given by
\begin{align*}
    \decayvar = \left(1 - 2\left(\sqrt{1 + (2L_0/m_c)} + 1\right)^{-1}\right)^{\frac{1}{2d - 1}}, C = \frac{2L_0}{m_c}\cdot \left(1 - 2\left(\sqrt{1 + (2L_0/m_c)} + 1\right)^{-1}\right)^{-1}.
\end{align*}
\end{theorem}

As a remark, the second inequality in Theorem \ref{thm:LTV-sensitivity} implies that the first inequality holds when $F$ is the indicator function of the origin. To see this, we only need to set $z = z' = 0$ in the second inequality.

Theorem \ref{thm:LTV-sensitivity} allows us to bound the distance between any two trajectories so long as they can be expressed as the optimal solutions of the optimization problem \eqref{opt:without-terminal} or \eqref{opt:with-terminal}. For example, to bound the norm of each state in the predictive trajectory $\tilde{\psi}_{t}^\horizonvar(x, \zeta; F)$, we only need to set $x' = 0, \zeta' = 0$ in the first inequality because an all zero trajectory can be expressed as $\tilde{\psi}_{t}^\horizonvar(0, 0; F)$. The formal statement of this result can be found in Appendix \ref{Appendix:coro:opt-stable}.

\section{Performance Guarantees for Predictive Control}
\label{sec:guarantees}


We now demonstrate the power of the perturbation approach in Section \ref{sec:perturbation-ltv} by obtaining bounds on regret and competitive ratio.
The key intuition behind our analysis is the following: at time step $t$, if the predictive controller with prediction window $k$ is given the knowledge of $x_t^*$ and $x_{t+k}^*$, it can fully recover the offline optimal states and control inputs for the future $k$ time steps, $x_{t+1:t+k}^*$ and $u_{t:t+k-1}^*$, from $\psi_t^k(x_t^*, w_{t:t+k-1}, x_{t+k}^*)$. However, without the knowledge of the offline optimal states, the predictive controller solves $\psi_t^k(x_t, w_{t:t+k-1}, x_{t+k})$ instead, where $x_{t+k}$ is implicitly determined by the $k$-th predictive state of $\tilde{\psi}_t^k(x_t, w_{t:t+k-1}; F)$. We overcome this gap with our perturbation approach (specifically, Theorem \ref{thm:LTV-sensitivity} and its corollaries in Appendix \ref{Appendix:coro:opt-stable} and \ref{Appendix:smooth-cost}), which allows us to bound the distance between the controller's trajectory and the offline optimal trajectory.

\subsection{Dynamic Regret}


We first bound the dynamic regret of predictive control.  For this analysis, a key observation is that the offline optimal trajectory is given by
$x^* = \tilde{\psi}_{0}^{T}\left(x_0, w_{0:T-1}; 0\right)_{y_{1:T}}.$ Furthermore, the optimal trajectory starting at time step $t$ with state $x_t$ is equivalent to the trajectory of predictive control with prediction window $(T-t)$ and no terminal cost, i.e. $\tilde{\psi}_{t}^{T-t}\left(x_t, w_{t:T-1}; 0\right)_{y_{1:T-t}}$.
Using Theorem \ref{thm:LTV-sensitivity}, we can bound the change in decision points against the change in prediction window $k$. Lemma \ref{lemma:one-step-diff} formalizes this:

\begin{lemma}\label{lemma:one-step-diff}
    Suppose the assumptions of Theorem \ref{thm:LTV-sensitivity} hold, and let $\decayvar, C$ be the decay rate and constants defined in Theorem \ref{thm:LTV-sensitivity}. For any integers $\horizonvar, h$ such that $\horizonvar \geq h \geq 1$ and time step $t < T - \horizonvar$, we have
    \[\norm{\tilde{\psi}_{t}^\horizonvar \left(x_t, w_{t:t+\horizonvar-1}; F\right)_{y_h} - \tilde{\psi}_{t}^{\horizonvar +1}\left(x_t, w_{t:t+\horizonvar}; F\right)_{y_h}} \leq 2 C \decayvar^{p-h}\left(C\decayvar^\horizonvar \norm{x_t} + \frac{2 C}{1 - \decayvar} \sup_{0\leq \tau \leq T-1}\norm{w_\tau}\right).\]
\end{lemma}

By cumulatively summing up the bounded difference in Lemma \ref{lemma:one-step-diff} and applying Theorem \ref{thm:LTV-sensitivity}, we can show that at state $x_t$ at time step $t$, the predictive controller picks a near-optimal control action $u_t$. Specifically, the distance between the predictive controller's next state $x_{t+1}$ and $\tilde{\psi}_{t}^{T-t}\left(x_t, w_{t:T-1}; 0\right)_{y_{1}}$ is in the order of $O(\decayvar^k)$, where $\decayvar$ is the decay rate of perturbation impact defined in Theorem \ref{thm:LTV-sensitivity}. From here, we can derive an $O(\decayvar^k)$ upper bound on the distance between the algorithm's trajectory and the offline optimal trajectory, which leads to the regret bound in Theorem \ref{thm:distance-MPC-infty}.


\begin{theorem}\label{thm:distance-MPC-infty}
Suppose $\norm{w_t} \leq D$ for some constant $D$ at each time step $t$. Suppose the assumptions of Theorem \ref{thm:LTV-sensitivity} hold, and let $\decayvar, C, L_0$ be the decay rate and constants defined in Theorem \ref{thm:LTV-sensitivity}. If prediction window $k \ge d$ is sufficiently large, such that
\begin{equation}\label{equ:regret-condition}
    k \geq 1 + \log\left(\frac{1}{1 - \delta}\cdot C \left(\frac{2C}{1 - \decayvar} + \decayvar\right)\right) \bigg/ \log\left(\frac{1}{\decayvar}\right)
\end{equation}
for some positive constant $\delta \in (0, 1)$, then the trajectory of $PC_k$ satisfies:
\begin{enumerate}
    \item (Input-to-state Stability) The norm of each state $x_t$ is upper bounded by
    \begin{equation*}
        \norm{x_t} \leq 
        \begin{cases}
            \frac{C}{\delta}\cdot (1 - \delta)^{\max(0, t - k)} \norm{x_0} + \frac{2 C}{\delta (1 - \decayvar)}\left(1 + \frac{2 C}{1 - \decayvar}\right) D & \text{ if }0 < t \leq T - k\\
            \frac{C^2}{\delta}\cdot (1 - \delta)^{T - 2k} \decayvar^{t+k-T}\norm{x_0} + \left(\frac{2 C^2}{\delta (1 - \decayvar)}\left(1 + \frac{2 C}{1 - \decayvar}\right) + \frac{2C}{1 - \decayvar}\right)D & \text{ if } T - k < t \leq T.
        \end{cases}
    \end{equation*}
    \item (Dynamic Regret) The dynamic regret of $PC_k$ is upper bounded by
    \[cost(PC_k) - cost(OPT) = O\left(\left(D + \frac{\decayvar^k(\norm{x_0} + D)}{\delta}\right)^2 \decayvar^{k} T + \lambda^k \norm{x_0}^2\right),\]
    where the notation hides quantities $a, b', \ell_f, \ell_c,$ $C, 1/(1 - \lambda)$ and $L_0$.
\end{enumerate}
\end{theorem}
An implication of Theorem \ref{thm:distance-MPC-infty} is that to obtain $o(1)$ dynamic regret when the norm of disturbances are uniformly upper bounded, it suffices to use a prediction window of length $\Theta(\log{T})$. This parallels the result shown in \cite{yu2020power}, although in a more general setting.

\subsection{Competitive Ratio}

We now focus on bounding the competitive ratio of predictive control. Here. we study a special case of the predictive control algorithm $PC_k$ (Algorithm \ref{alg:pc}), where the terminal cost $F$ is the indicator function of the origin, i.e., 
\[F(x) = \begin{cases}
0 & \text{ if } x = 0,\\
+\infty & \text{ otherwise}.
\end{cases}\]
This indicator terminal cost function enforces that the last predictive state in every predictive trajectory must be the origin, which can be achieved when $k \geq d$ because of the controllability (Assumption \ref{assump:LTV}).

The idea of using the strict terminal constraint $x = 0$, which is equivalent to our indicator terminal cost $F$, was first proposed in the MPC literature as the simplest way to guarantee the recursive feasibility and stability when there are state and control constraints \cite{borrelli2017predictive}. A similar technique of moving towards to the minimizer is also considered in the SOCO literature \cite{lin2020online, argue2020dimension}, which has been proved to be a simple and robust way to achieve constant competitive ratio even in non-convex settings \cite{lin2020online}. Technically, we require the terminal state of every predictive trajectory $\tilde{\psi}_t^k\left(x_t, w_{t:t+k-1}; F\right)$ to be $0$ because it allows us to bound its distance with offline optimal state $x_{t+k}^*$ by the offline optimal state cost at time $t+k$.



Our proof for the competitive ratio result is inspired by the widely-used potential method (see \cite{bansal2019potential} for a survey). We define the potential as the squared distance between the algorithm's trajectory and the offline optimal trajectory, i.e.,
$\phi_t = \norm{x_t - x_t^*}^2.$
The same potential has been adopted by many previous work in the SOCO literature, and they show competitive ratios by comparing $\phi_t$ with $\phi_{t-1}$ \cite{goel2019beyond, goel2018smoothed,shi2020online}. Our proof technique is different from previous work in that we compare $\phi_t$ with a weighted sum of $\phi_1, \ldots, \phi_{t-1}$, where the sum of the weights is less than $1$ for sufficiently large prediction window $k$. The resulting inequalities allow us to upper bound $\sum_{t=1}^{T-k} \phi_t$ by $O(\decayvar^{2k})$ times the offline optimal cost, which can further be used to derive the competitive ratio bound. 


\begin{theorem}\label{thm:new-competitive-ratio}
Consider the case where the terminal cost $F$ in $PC_k$ is the indicator function of the origin. Suppose the assumptions of Theorem \ref{thm:LTV-sensitivity} hold, and let $\decayvar, C, L_0$ be the decay rate and constants defined in Theorem~\ref{thm:LTV-sensitivity}. If the prediction window length $k \geq d$ is sufficiently large such that for some constant $\varepsilon \in (0, 1)$,
\[k \geq \ln\left(\frac{6 C^6}{(1 - \varepsilon) \lambda^2 (1 - \lambda)^2 (1 - \lambda^2)^2}\right)/ \left(4\ln(1/\lambda)\right),\]
then the competitive ratio of $PC_k$ (Algorithm \ref{alg:pc}) is \[1 + \lambda^k \cdot \left(1 + \frac{24 C^4 (C + 1)^2 \left(2\ell_f + 4(b')^2 \ell_c + 4 a^2 (b')^2 \ell_c + L_0 + \ell_f\right)}{\epsilon \cdot \lambda^4 (1 - \lambda)^2(1 - \lambda^2)^2 m_f}\right).\]
\end{theorem}

Note that, for a fixed choice of the constant $\varepsilon$, the competitive ratio is on the order of $1 + O(\decayvar^k)$ as the length of prediction $k$ tends to infinity. One potential line of future work is to better understand the role of terminal cost. It may be possible to relax the assumptions on the terminal cost so that the controller's predictive trajectory does not have to return to the origin.


\bibliographystyle{unsrtnat}
\bibliography{main.bib}

\clearpage

\appendix

\section{Definitions and Notations}\label{Appendix:def-and-notations}
\label{sec:standard-notations}
\begin{definition}
We use the follow convention on linear algebra:
\begin{enumerate}
    \item $\norm{\cdot}$ denotes the (Euclidean) $2$-norm for vectors and the induced $2$-norm for matrices:
    \begin{align*}
        \norm{v} &= \sqrt{v_1^2 + v_2^2 + \cdots + v_n^2}, \; v \in \bR^n \\
        \norm{A} &= \sup_{v \in \bR^n \setminus \{0\}} \frac{\norm{Ax}}{\norm{x}}, \; A \in \bR^{m \times n};
    \end{align*}
    \item $\sigma(A)$ is the collection of singular values of a matrix $A$, also known as the singular spectrum;
    \item $\sigma_{\min}(A)$ denotes the smallest singular value of a matrix $A$;
    \item $A \succeq 0$ indicates that a matrix $A$ is positive semi-definite.
\end{enumerate}
\end{definition} 
The notions of strong convexity and smoothness are used throughout this paper:
\begin{definition}
    A real-valued function $g : \bR^n \to \bR$ is called $\ell$-strongly smooth if 
    \[ g(y) \leq g(x) + \langle \nabla g(x), y-x \rangle + \frac{\ell}{2} \norm{y - x}_2^2\]
    holds for any $x, y \in \bR^n$, and is called $m$-strongly convex if 
    \[ g(y) \geq g(x) + \langle \nabla g(x), y-x \rangle + \frac{m}{2} \norm{y - x}_2^2 \]
    holds for any $x, y \in \bR^n$. Here $\langle \cdot, \cdot \rangle$ denotes the standard inner product of vectors.
\end{definition}
For the regret bound, we require the terminal cost to be a K-function, the definition of which is given below.
\begin{definition}
    A function $g: \bR_{\geq 0} \to \bR_{\geq 0}$ is said to be a K-function (or belongs to class K), if it is continuous, strictly increasing, and satisfies $g(0) = 0$.
\end{definition}

\section{Proof of Theorem \ref{thm:SOCO-sensitivity}}\label{Appendix:thm:SOCO-sensitivity}
Before we show Theorem \ref{thm:SOCO-sensitivity}, we will first show a result in Lemma \ref{thm:band-mat-inverse-3} about the exponential decay phenomena in the inverse of a (block) banded matrix: the magnitudes of elements decay exponentially with respect to their distance to the main diagonal. Our Lemma \ref{thm:band-mat-inverse-3} generalizes Proposition 2.2 in \cite{demko1984decay} to consider block matrices and additive terms on the main diagonal. We need this generalization because we want to apply this result to the Hessian of the objective function in \eqref{equ:SOCO-opt}. 
Other than matrix inverses and ordinary banded matrices, similar exponential decay results for matrices also exists for general analytic functions \cite{benzi2007decay} and graph-induced banded matrices \cite{shin2020decentralized}.

In Lemma \ref{thm:band-mat-inverse-3}, we will use the notation $A_{S_R, S_C}$ to denote the submatrix obtained by selecting the blocks indexed by some set $S_R \times S_C$ while preserving their relative order. Specifically, consider a matrix $A \in \bR^{\omega n \times \omega n}$ formed by $\omega \times \omega$ blocks $A_{i,j} \in \bR^{n \times n}$. Let $i_1 < \cdots < i_{|S_R|}$ be the elements in $S_R \subseteq \{1, \ldots, \omega \}$, and $j_1 < \cdots < j_{|S_C|}$ be the elements in $S_C \subseteq \{1, \ldots, \omega \}$, both in ascending order. Then $A_{S_R, S_C} \in \bR^{|S_R|n \times |S_C|n}$ is defined as a block matrix
\[ A_{S_R, S_C} := \begin{bmatrix}
    A_{i_1, j_1} & A_{i_1, j_2} & \cdots & A_{i_1, j_{|S_C|}} \\
    A_{i_2, j_1} & A_{i_2, j_2} & \cdots & A_{i_2, j_{|S_C|}} \\
    \vdots & \vdots & \ddots & \vdots \\
    A_{i_{|S_R|}, j_1} & A_{i_{|S_R|}, j_2} & \cdots & A_{i_{|S_R|}, j_{|S_C|}} \\
\end{bmatrix}. \]
For a diagonal block matrix $D = diag(D_1, \ldots, D_{\omega})$ and a set $S \subseteq \{1, \ldots, \omega \}$, we use the shorthand notation $D_S := diag\left(D_{i_1}, D_{i_2}, \ldots, D_{i_{\abs{S}}}\right)$, where $i_1 < \ldots < i_{\abs{S}}$ are the elements in $S$.

\begin{lemma}\label{thm:band-mat-inverse-3}
Suppose $A$ is a positive definite matrix in $\mathbb{S}^{\omega n}$ formed by $\omega \times \omega$ blocks $A_{i,j} \in \mathbb{R}^{n\times n}$. Assume that $A$ is $q$-banded for an even positive integer $q$; i.e.,
\[A_{i,j} = 0, \forall \abs{i - j} > q/2.\]
Let $[a_0, b_0]$ ($b_0 > a_0 > 0$) be the smallest interval containing $\sigma(A)$, the spectrum of $A$. Suppose $D = diag(D_1, \ldots, D_\omega)$, where $D_i \in \mathbb{S}^n$ is positive semi-definite. Let $M = \left((A+D)^{-1}\right)_{S_R,S_C}$ as defined above, where $S_R, S_C \subseteq \{1, \ldots, \omega \}$. 
Then we have $\norm{M} \leq C \gamma^{\hat{d}}$, where
\[C = \frac{2}{a_0}, \gamma = \left(\frac{\sqrt{cond(A)} - 1}{\sqrt{cond(A)} + 1}\right)^{2/q}, \hat{d} = \min_{i \in S_R, j \in S_c} \abs{i - j}.\]
Here $cond(A) = b_0/a_0$ denotes the condition number of matrix $A$.
\end{lemma}

\begin{proof}[Proof of Lemma \ref{thm:band-mat-inverse-3}]
    We first prove the lemma for the the special case where $D = 0$. 
    
    For the case $\hat{d} \neq 0$, write $\hat{d} = \upsilon q/2 + \kappa$ for integers $\upsilon, \kappa$ satisfying $\upsilon \geq 0, 1 \leq \kappa \leq q/2$. Following the same approach as the proof of Proposition 2.2 in \cite{demko1984decay}, we see that there exists a polynomial $p_\upsilon$ of degree $\upsilon$, where
    \[\norm{A^{-1} - p_\upsilon(A)} \leq \frac{1}{a_0}\cdot \frac{\left(1 + \sqrt{cond(A)}\right)^2}{2cond(A)} \gamma^{\hat{d}}\leq C \gamma^{\hat{d}},\]
    where the last inequality holds because $cond(A) \geq 1$.
    
    Since $p_v$ has degree $v < \frac{2 \hat{d}}{q}$ and $A$ is $q$-banded, the matrix $p_{\upsilon}(A)$ satisfies $\left(p_{\upsilon}(A)\right)_{i, j} = 0$ for any $i \in S_R$ and $j \in S_C$.
    We then obtain
    \[\norm{P} = \norm{\left(A^{-1}\right)_{S_R,S_C}} = \norm{\left(A^{-1} - p_{\upsilon}(A)\right)_{S_R,S_C}} \leq \norm{A^{-1} - p_{\upsilon}(A)} \leq C \gamma^{\hat{d}},\]
    due to the fact that the 2-norm of a submatrix cannot be larger than that of the original matrix.
        
    For the case $\hat{d} = 0$, as $\norm{P} = \norm{\left(A^{-1}\right)_{S_R, S_C}} \leq \norm{A^{-1}} = \frac{1}{a_0} \leq C$, the result trivially holds.
    
    Now we show the general case (where $D_i \succeq 0$ for $1 \le i \le n$) through a reduction to the special case. Define a positive definite matrix $N := (a_0 I + D) \in \mathbb{S}^{n\omega}$, and then define matrix $H \in \mathbb{S}^{n\omega}$ as
    \[H := N^{-\frac{1}{2}} (A + D) N^{-\frac{1}{2}}.\]
    We start by showing that $I \preceq H \preceq \frac{b_0}{a_0}\cdot I$. For any $x \in \mathbb{R}^{n\omega}$, we observe that
    \begin{align*}
        x^{\top} H x &= x^{\top} N^{-\frac{1}{2}}A N^{-\frac{1}{2}}x + x^{\top} N^{-\frac{1}{2}}D N^{-\frac{1}{2}}x\\
        &\geq x^{\top} N^{-\frac{1}{2}} a_0 I N^{-\frac{1}{2}}x + x^{\top} N^{-\frac{1}{2}}D N^{-\frac{1}{2}}x\\
        &= x^{\top} N^{-\frac{1}{2}} (a_0 I + D) N^{-\frac{1}{2}}x\\
        &= \norm{x}^2.
    \end{align*}
    For the other inequality, note that $0 \prec N^{-1} \preceq \frac{1}{a_0} \cdot I$, so we have
    \begin{align*}
        x^{\top} H x &= x^{\top} N^{-\frac{1}{2}}A N^{-\frac{1}{2}}x + x^{\top} N^{-\frac{1}{2}}D N^{-\frac{1}{2}}x\\
        &\leq x^{\top} N^{-\frac{1}{2}} b_0 I N^{-\frac{1}{2}}x + x^{\top} N^{-\frac{1}{2}}D N^{-\frac{1}{2}}x\\
        &= x^{\top} N^{-\frac{1}{2}} (a_0 I + D) N^{-\frac{1}{2}}x + (b_0 - a_0) x^{\top} N^{-1} x\\
        &\leq \norm{x}^2 + \frac{b_0 - a_0}{a_0} \cdot \norm{x}^2\\
        &= \frac{b_0}{a_0}\cdot \norm{x}^2.
    \end{align*}
    Thus $I \preceq H \preceq \frac{b_0}{a_0} \cdot I$, which gives $cond(H) \leq \frac{b_0}{a_0} = cond(A)$. Note that $H$ is also $q$-banded, so we can apply the result of the special case ($D_i = 0, i=1,\cdots,n$) to obtain that
    \[\norm{(H^{-1})_{S_R,S_C}} \leq 2 \gamma_H^{\hat{d}} \leq 2 \gamma^{\hat{d}},\]
    where $\gamma_H = \left(\frac{\sqrt{cond(H)} - 1}{\sqrt{cond(H)} + 1}\right)^{2/q} \leq \gamma$. Using this inequality, we conclude that
    \begin{align*}
        \norm{P} = \norm{((A + D)^{-1})_{S_R,S_C}} &= \norm{\left(N^{-\frac{1}{2}} H^{-1} N^{-\frac{1}{2}}\right)_{S_R,S_C}}\\
        &\leq \norm{(a_0 I + D_{S_R})^{-\frac{1}{2}}} \cdot \norm{(H^{-1})_{S_R,S_C}} \cdot \norm{(a_0 I + D_{S_C})^{-\frac{1}{2}}}\\
        &\leq \frac{1}{a_0}\norm{(H^{-1})_{S_R,S_C}}\\
        &\leq C \gamma^{\hat{d}}.
    \end{align*}
    Here we apply the fact that $\norm{(a_0 I + D_{S})^{-\frac{1}{2}}} \leq \frac{1}{\sqrt{a_0}}$ since $D_S \succeq 0$.
\end{proof}

Now we return to the proof of Theorem \ref{thm:SOCO-sensitivity}. The intuition behind the proof is discussed in Section \ref{sec:perturbation-soco}.

Let $e = [e_0^\top, \pi^\top, e_p^\top]^\top$ be a vector where $e_0, e_p \in \mathbb{R}^n$ and
\[\pi = [\pi_0, \pi_1, \ldots, \pi_{p-1}],\]
for $\pi_i \in \mathbb{R}^{r}, i = 0, 1, \ldots, p-1$.  
Let $\theta$ be an arbitrary real number. Define function $\hat{h}: \mathbb{R}^{(p - 1)\times n} \times \mathbb{R}^n \times \mathbb{R}^{p \times r} \times \mathbb{R}^n \to \mathbb{R}_+$ as
\[\hat{h}(\hat{x}_{1:p-1}, \hat{x}_0, \hat{w}_{0:p-1}, \hat{x}_p) = \sum_{\tau = 1}^{p-1} \hat{f}_{\tau}(\hat{x}_{\tau}) + \sum_{\tau = 1}^p \hat{c}_{\tau}(\hat{x}_{\tau}, \hat{x}_{\tau - 1}, \hat{w}_{\tau-1}).\]
To simplify the notation, we use $\hat{\zeta}$ to denote the tuple of system parameters, i.e.,
\[\hat{\zeta} := (\hat{x}_0, \hat{w}_{0:p-1}, \hat{x}_p).\]

From our construction, we know that $\hat{h}$ is $\mu$-strongly convex in $\hat{x}_{1:p-1}$, so we use the decomposition $\hat{h} = \hat{h}_a + \hat{h}_b$, where
\begin{align*}
    \hat{h}_a(\hat{x}_{1:p-1}, \hat{\zeta}) &= \sum_{\tau = 1}^{p-1} \frac{\mu}{2}\norm{\hat{x}_{\tau}}^2 + \sum_{\tau = 1}^p \hat{c}_{\tau}(\hat{x}_{\tau}, \hat{x}_{\tau - 1}, \hat{w}_{\tau - 1}),\\
    \hat{h}_b(\hat{x}_{1:p-1}, \hat{\zeta}) &= \sum_{\tau = 1}^{p-1} \left(\hat{f}_{\tau}(\hat{x}_{\tau}) - \frac{\mu}{2}\norm{\hat{x}_{\tau}}^2\right).
\end{align*}

Since $\hat{\psi}(\hat{\zeta} + \theta e)$ is the minimizer of convex function $\hat{h}(\cdot, \hat{\zeta} + \theta e)$, we see that
\[\nabla_{\hat{x}_{1:p-1}} \hat{h}(\hat{\psi}(\hat{\zeta} + \theta e), \hat{\zeta} + \theta e) = 0.\]
Taking the derivative with respect to $\theta$ gives that
\begin{align*}
    \nabla_{\hat{x}_{1:p-1}}^2 \hat{h}(\hat{\psi}(\hat{\zeta} + \theta e), \hat{\zeta} + \theta e) \frac{d}{d\theta}\hat{\psi}(\hat{\zeta} + \theta e) ={}& - \nabla_{\hat{x}_0} \nabla_{\hat{x}_{1:p-1}} \hat{h}(\hat{\psi}(\hat{\zeta} + \theta e), \hat{\zeta} + \theta e) e_0\\
    &- \nabla_{\hat{x}_p} \nabla_{\hat{x}_{1:p-1}} \hat{h}(\hat{\psi}(\hat{\zeta} + \theta e), \hat{\zeta} + \theta e) e_p\\
    &- \sum_{\tau = 0}^{p-1}\nabla_{w_\tau} \nabla_{\hat{x}_{1:p-1}} \hat{h}(\hat{\psi}(\hat{\zeta} + \theta e), \hat{\zeta} + \theta e) \pi_\tau.
\end{align*}
To simplify the notation, we define
\begin{align*}
    M &:= \nabla_{\hat{x}_{1:p-1}}^2 \hat{h}(\hat{\psi}(\hat{\zeta} + \theta e), \hat{\zeta} + \theta e), \text{which is a }(p-1)\times (p-1) \text{ block matrix},\\
    R^{(0)} &:= - \nabla_{\hat{x}_0} \nabla_{\hat{x}_{1:p-1}} \hat{h}(\hat{\psi}(\hat{\zeta} + \theta e), \hat{\zeta} + \theta e), \text{which is a }(p-1)\times 1 \text{ block matrix},\\
    R^{(p)} &:= - \nabla_{\hat{x}_p} \nabla_{\hat{x}_{1:p-1}} \hat{h}(\hat{\psi}(\hat{\zeta} + \theta e), \hat{\zeta} + \theta e), \text{which is a }(p-1)\times 1 \text{ block matrix},\\
    K^{(\tau)} &:= - \nabla_{w_\tau} \nabla_{\hat{x}_{1:p-1}} \hat{h}(\hat{\psi}(\hat{\zeta} + \theta e), \hat{\zeta} + \theta e), \forall 0 \leq \tau \leq p-1, \text{which are }(p-1)\times 1 \text{ block matrices},
\end{align*}
where in $M$, $R^{(0)}$, $R^{(p)}$, the block size is $n\times n$; in $K^{(\tau)}$, the block size is $n\times r$. Hence we can write
\[\frac{d}{d\theta}\hat{\psi}(\hat{\zeta} + \theta e) = M^{-1} \left(R^{(0)} e_0 + R^{(p)} e_p + \sum_{\tau=0}^{p-1} K^{(\tau)} \pi_\tau\right).\]
Recall that $R^{(0)}, R^{(p)}$ are $(p-1) \times 1$ block matrices with block size $n \times n$; $\{K^{(\tau)}\}_{0\leq \tau \leq p-1}$ are $(p-1) \times 1$ block matrices with block size $n \times r$. For $R^{(0)}$ and $K^{(0)}$, only the $(1, 1)$-th blocks are non-zero. For $R^{(p)}$ and $K^{(p-1)}$, only the $(p-1, 1)$-th blocks are non-zero. For $K^{(\tau)}, \tau = 1, \ldots, p-2$, only the $(\tau, 1)$-th and $(\tau+1, 1)$-th blocks are non-zero. Hence we see that
\begin{align*}
    \frac{d}{d\theta}\hat{\psi}(\hat{\zeta} + \theta e)_h ={}& (M^{-1})_{h, 1} R^{(0)}_{1, 1} e_0 + (M^{-1})_{h, p-1} R^{(p)}_{p-1, 1} e_p\\
    &+ (M^{-1})_{h, 1} K^{(0)}_{1, 1} \pi_0 + (M^{-1})_{h, p-1} K^{(p-1)}_{p-1, 1} \pi_{p-1}\\
    &+ \sum_{\tau=1}^{p-2} (M^{-1})_{h, \tau:\tau+1} K^{(\tau)}_{\tau:\tau+1, 1} \pi_\tau.
\end{align*}
Since the switching costs $c_\tau(\cdot, \cdot, \cdot), \tau = 1, \ldots, p$ are $\ell$-strongly smooth, we know that the norms of
\[R^{(0)}_{1, 1}, R^{(p)}_{p-1, 1}, K^{(0)}_{1, 1}, K^{(p-1)}_{p-1, 1}, \text{ and } \{K^{(\tau)}_{\tau:\tau+1, 1}\}_{1\leq \tau \leq p-2}\]
are all upper bounded by $\ell$. 
Taking norms on both sides gives that
\begin{align}\label{equ:thm:SOCO-sensitivity:e1}
    \norm{\frac{d}{d\theta}\hat{\psi}(\hat{\zeta} + \theta e)_h} \leq{}& \ell\norm{(M^{-1})_{h, 1}} \norm{e_0} + \ell\norm{(M^{-1})_{h, p-1}} \Vert e_p \Vert \nonumber\\
    &+ \ell\norm{(M^{-1})_{h, 1}} \norm{\pi_0} + \ell\norm{(M^{-1})_{h, p-1}} \Vert \pi_{p-1} \Vert \nonumber\\
    &+ \sum_{\tau=1}^{p-2} \ell\norm{(M^{-1})_{h, \tau:\tau+1}} \norm{\pi_\tau}.
\end{align}
Note that $M$ can be decomposed as $M = M_a + M_b$, where
\begin{align*}
    M_a &:= \nabla_{1:p-1}^2 \hat{h}_a(\hat{\psi}(\hat{\zeta} + \theta e), \hat{\zeta} + \theta e),\\
    M_b &:= \nabla_{1:p-1}^2 \hat{h}_b(\hat{\psi}(\hat{\zeta} + \theta e), \hat{\zeta} + \theta e).
\end{align*}
Since $M_a$ is block tri-diagonal and satisfies $(\mu + 2\ell) I \succeq M_a \succeq \mu I$, and $M_b$ is block diagonal and satisfies $M_b \succeq 0$, we obtain the following using Lemma \ref{thm:band-mat-inverse-3}: 
\[\norm{(M^{-1})_{h, 1}} \leq \frac{2}{\mu} \decayvar_0^{h-1}, \norm{(M^{-1})_{h, p-1}} \leq \frac{2}{\mu} \decayvar_0^{p-h-1}, \text{ and } \norm{(M^{-1})_{h, \tau:\tau+1}} \leq \frac{2}{\mu} \decayvar_0^{\abs{h - \tau} - 1},\]
where $\decayvar_0 :={(\sqrt{cond(M_a)} - 1)}/{(\sqrt{cond(M_a)} + 1)} = 1 - 2 \cdot \left(\sqrt{1 + (2\ell/\mu)} + 1\right)^{-1}$.

Substituting this into \eqref{equ:thm:SOCO-sensitivity:e1}, we see that
\[\norm{\frac{d}{d\theta}\hat{\psi}(\hat{\zeta} + \theta e)_h} \leq C_0\left(\decayvar_0^{h-1}\norm{e_0} + \sum_{\tau=0}^{p-1} \decayvar_0^{\abs{h - \tau} - 1}\norm{\pi_\tau} + \decayvar_0^{p-h-1} \Vert e_p \Vert \right),\]
where $C_0 = (2\ell)/\mu$.

Finally, by integration we have
\begin{align*}
    \norm{\hat{\psi}(\hat{\zeta})_h - \hat{\psi}(\hat{\zeta} + e)_h} ={}& \norm{\int_0^1 \frac{d}{d\theta}\hat{\psi}(\hat{\zeta} + \theta e)_h d\theta}\\
    \leq{}& \int_0^1 \norm{\frac{d}{d\theta}\hat{\psi}(\hat{\zeta} + \theta e)_h} d\theta\\
    \leq{}& C_0\left(\decayvar_0^{h-1}\norm{e_0} + \sum_{\tau=0}^{p-1} \decayvar_0^{\abs{h - \tau} - 1}\norm{\pi_\tau} + \decayvar_0^{p-h-1} \Vert e_p \Vert \right).
\end{align*}
This finishes the proof of Theorem \ref{thm:SOCO-sensitivity}. 

\section{Proof of Lemma \ref{lemma:implicit-switching-cost}}\label{Appendix:lemma:implicit-switching-cost}
\label{sec:proof-implicit-switching-cost}

Before showing Lemma \ref{lemma:implicit-switching-cost}, we first show a general result about the properties of the optimal solution and optimal value of an unconstrained optimization problem.

\begin{lemma}\label{lemma:smoothness-opt-value}
Suppose function $f(x, y)$ is convex and $L$-strongly smooth in $(x, y)$, $\mu$-strongly convex in $y$, and continuously differentiable. Define $y^*(x) := \argmin_y f(x, y)$ and $g(x) := \min_y f(x, y)$. Then, function $y^*$ is $\frac{L}{\mu}$-Lipschitz, and function $g$ is $\left(L + \frac{L^2}{\mu}\right)$-strongly smooth.
\end{lemma}

\begin{proof}[Proof of Lemma \ref{lemma:smoothness-opt-value}]
    Let $y^*(x) = \argmin_y f(x, y)$. This function is well-defined since the strong convexity of $f(x, y)$ in $y$ guarantees that $y^*(x)$ is unique. We see that for all $x, x'$,
    \[\nabla_y f(x, y^*(x)) = 0 \text{ and } \nabla_y f(x', y^*(x')) = 0.\]
    Using these equalities, we obtain 
    \begin{align*}
        0 ={}& \langle y^*(x) - y^*(x'), \nabla_y f(x, y^*(x)) - \nabla_y f(x', y^*(x'))\rangle\\
        ={}& \langle y^*(x) - y^*(x'), \nabla_y f(x, y^*(x)) - \nabla_y f(x, y^*(x'))\rangle\\
        &+ \langle y^*(x) - y^*(x'), \nabla_y f(x, y^*(x')) - \nabla_y f(x', y^*(x'))\rangle\\
        \geq{}& \mu \norm{y^*(x) - y^*(x')}^2 - \norm{y^*(x) - y^*(x')} \cdot \norm{\nabla_y f(x, y^*(x')) - \nabla_y f(x', y^*(x'))},
    \end{align*}
    where we use the fact that a $\mu$-strongly convex function $h$ satisfies
    \[\langle a - b, \nabla h(a) - \nabla h(b) \rangle \geq \mu \norm{a - b}^2,~ \forall a, b\]
    and the Cauchy-Schwarz inequality in the last inequality. Since $f$ is $L$-strongly smooth, we see that
    \begin{align*}
        \norm{y^*(x) - y^*(x')} \leq \frac{1}{\mu} \norm{\nabla_y f(x, y^*(x')) - \nabla_y f(x', y^*(x'))} \leq \frac{L}{\mu} \norm{x - x'},
    \end{align*}
    which implies function $y^*$ is $\frac{L}{\mu}$-Lipschitz.
    
    Note that the gradient of $g$ is given by
    \[\nabla g(x) = \nabla_x f(x, y^*(x)) + \nabla_y f(x, y^*(x)) \frac{\partial y^*(x)}{\partial x} = \nabla_x f(x, y^*(x)),\]
    because $\nabla_y f(x, y^*(x)) = 0$. Hence we obtain
    \begin{align*}
        \norm{\nabla g(x) - \nabla g(x')} \leq{}& \norm{\nabla_x f(x, y^*(x)) - \nabla_x f(x', y^*(x')) }\\
        \leq{}& \norm{\nabla_x f(x, y^*(x)) - \nabla_x f(x', y^*(x))} +\norm{\nabla_x f(x', y^*(x)) - \nabla_x f(x', y^*(x')) }\\
        \leq{}& L \norm{x - x'} + L \norm{y^*(x) - y^*(x')}\\
        \leq{}& \left(L + \frac{L^2}{\mu}\right)\norm{x - x'},
    \end{align*}
    where we use the $L$-strong smoothness of $f$ and the $\frac{L}{\mu}$-Lipschitzness of $y^*$.
\end{proof}

Besides, we also need the following lemma about the upper bound of the induced 2-norm of a block matrix.

\begin{lemma}\label{lemma:mat-norm-upper}
Suppose $A$ is a $\omega_1 \times \omega_2$ block matrix. Let $A_{ij}$ denote the $(i, j)$-th block of $A$, $1 \leq i \leq \omega_1, 1 \leq j \leq \omega_2$. The induced $2$-norm of $A$ is upper bounded by
\[\norm{A} \leq \left(\sum_{i=1}^{\omega_1} \sum_{j=1}^{\omega_2} \norm{A_{ij}}^2\right)^{\frac{1}{2}}.\]
\end{lemma}
\begin{proof}[Proof of Lemma \ref{lemma:mat-norm-upper}]
For unit vector $x$, we have the following:
\begin{align*}
    \norm{A x}^2 &= \sum_{i = 1}^{\omega_1} \norm{\sum_{j=1}^{\omega_2} A_{ij} x_j}^2\\
    &\leq \sum_{i=1}^{\omega_1}\left(\sum_{j=1}^{\omega_2} \norm{A_{ij}} \cdot \norm{x_j}\right)^2\\
    &\leq \sum_{i=1}^{\omega_1}\left(\sum_{j=1}^{\omega_2} \norm{A_{ij}}^2\right)\left(\sum_{j=1}^{\omega_2} \norm{x_j}^2\right)\\
    &= \sum_{i=1}^{\omega_1} \sum_{j=1}^{\omega_2} \norm{A_{ij}}^2.
\end{align*}
where we use the definition of the induced 2-norm in the first inequality and the Cauchy-Schwarz inequality in the second inequality.
\end{proof}

Now we come back to the proof of Lemma \ref{lemma:implicit-switching-cost}. To apply Lemma \ref{lemma:smoothness-opt-value}, we only need to show that $\xi_t^p(x, \zeta, z)$ can be formulated as an unconstrained minimization problem of a strongly smooth objective function which is also strongly convex in its optimization variable.

 To simplify the notation, we define the stacked state vector $y$, control vector $v$, and disturbance vector $\zeta$ as
    \[y = \begin{bmatrix}y_0\\ y_{1}\\ \vdots \\ y_{p}\end{bmatrix}, v = \begin{bmatrix}v_0\\ v_{1}\\ \vdots \\ v_{p-1}\end{bmatrix}, \zeta = \begin{bmatrix}\zeta_0\\ \zeta_{1}\\ \vdots \\ \zeta_{p-1}\end{bmatrix}.\]
    Recall that the transition matrix $\Phi(t_2, t_1)$ is defined as
    \[\Phi(t_2, t_1) := \begin{cases}
        A_{t_2 - 1} A_{t_2 - 2} \cdots A_{t_1} & \text{ if }t_2 > t_1\\
        I & \text{ if } t_2 \leq t_1
    \end{cases}.\]
    Using these notations, we can express the state vector $y$ as an affine function of initial state $x$, control $v$, and disturbance $\zeta$:
    \begin{equation}\label{lemma:implicit-switching-cost:e0}
        y = S^x x + S^v v + S^\zeta \zeta,
    \end{equation}
    where
    \begin{equation*}
        S^\zeta := \begin{bmatrix}
        0 & 0 & \cdots & 0\\
        \Phi(t+1, t+1) & 0 & \cdots & 0\\
        \Phi(t+2, t+1) & \Phi(t+2, t+2) & \cdots & 0\\
        \vdots & \vdots & \ddots & \vdots\\
        \Phi(t+p, t+1) & \Phi(t+p, t+2) & \cdots & \Phi(t+p, t+p)
        \end{bmatrix}, S^x = \begin{bmatrix}
        \Phi(t, t)\\
        \Phi(t+1, t)\\
        \Phi(t+2, t)\\
        \vdots\\
        \Phi(t+p, t)
        \end{bmatrix},
    \end{equation*}
    and $S^v = S^\zeta \cdot diag(B_t, \ldots, B_{t+p-1})$.
    
    For simplicity, we use the shorthand notation $M := M(t, p)$ for the controllability matrix and
    \[R^\zeta := [\Phi(t+p, t+1), \Phi(t+p, t+2), \ldots, \Phi(t+p, t+p)]\]
    throughout the proof. Since $p$ is greater than the controllability index $d$, we know $M$ has full row rank. The dynamical constraints for \eqref{equ:switching-cost}, which is identical to the constraints of \eqref{opt:with-terminal}, can be written as
    \[M v = z - \Phi(t+p, t) x - R^\zeta \zeta.\]
    Because $M$ has full row rank, we let $M^\dagger = M^\top \left(M M^\top\right)^{-1}$ be the Moore-Penrose pseudo-inverse of $M$. Let $V \in \mathbb{R}^{(mp) \times (mp - n)}$ be a matrix whose columns constitute an orthonormal basis of $ker(M)$. Then, we can express any feasible control vector $v$ as
    \begin{equation}\label{lemma:implicit-switching-cost:e1}
        v = M^\dagger \left(z - \Phi(t+p, t) x - R^\zeta \zeta\right) + V r,
    \end{equation}
    where $r$ is a free variable that can take any value in $\mathbb{R}^{mp - n}$. 
    
    Let $G$ denote the objective function of the optimization problem induced by $\xi_t^p$, i.e.,
    \[G(y, v) := \left(\sum_{\tau=1}^{p-1} f_{t + \tau}(y_\tau)+ c_{t + \tau}(v_{\tau - 1})\right) + c_{t+p}(v_{p-1}).\]
    Since we can express the state vector $y$ and control vector $v$ as linear functions of $x, z, \zeta$ and $r$, we can write the switching cost \eqref{equ:switching-cost} as an unconstrained optimization problem
    \begin{equation}\label{lemma:implicit-switching-cost:e2}
        \min_{r \in \mathbb{R}^{mp - n}} G(y(x, z, \zeta, r), v(x, z, \zeta, r)),
    \end{equation}
    where functions $y(x, z, \zeta, r)$ and $v(x, z, \zeta, r)$ are determined by
    \begin{equation}\label{lemma:implicit-switching-cost:e3}
        \begin{bmatrix}
            y\\
            v
        \end{bmatrix} = \begin{bmatrix}
            S^x - S^v M^\dagger \Phi(t+p, t) & S^v M^\dagger & S^\zeta - S^v M^\dagger R^\zeta & S^v V\\
            - M^\dagger \Phi(t+p, t) & M^\dagger & -M^\dagger R^\zeta & V
        \end{bmatrix} \cdot \begin{bmatrix}
            x\\
            z\\
            \zeta\\
            r
        \end{bmatrix}.
    \end{equation}
    Note that, when $a \not = 1$, it follows from Lemma \ref{lemma:mat-norm-upper} that
    \begin{align*}
        \Vert S^\zeta \Vert \leq \left(\sum_{i=1}^p \sum_{j=1}^i \norm{\Phi(t+i, t+j)}^2\right)^{\frac{1}{2}} \leq \left(\sum_{i=1}^p \sum_{j=1}^i a^{2(i-j)}\right)^{\frac{1}{2}} = \frac{\sqrt{a^{2p+2} - (p+1)a^2 + p}}{\abs{a^2 - 1}};
    \end{align*}
    similarly, by Lemma \ref{lemma:mat-norm-upper} we also have
    \[\norm{S^x} \leq \sqrt{\frac{a^{2p+2} - 1}{a^2 - 1}},~
      \norm{M^\dagger} \leq \frac{b}{\sigma^2}\cdot \sqrt{\frac{a^{2p} - 1}{a^2 - 1}},~
      \norm{S^v} \leq b\Vert S^\zeta \Vert,~
      \Vert R^\zeta \Vert \leq \sqrt{\frac{a^{2p} - 1}{a^2 - 1}} \leq \frac{a^p - 1}{a - 1}.\]
    Note that Lemma \ref{lemma:mat-norm-upper} implies a slightly looser bound $\norm{A} \leq \sum_{i=1}^{\omega_1} \sum_{j=1}^{\omega_2} \norm{A_{ij}}$, so we have
    \begin{equation}\label{lemma:implicit-switching-cost:e4}
        \norm{\begin{bmatrix}
            S^x - S^v M^\dagger \Phi(t+p, t) & S^v M^\dagger & S^\zeta - S^v M^\dagger R^\zeta & S^v V\\
            - M^\dagger \Phi(t+p, t) & M^\dagger & -M^\dagger R^\zeta & V
        \end{bmatrix}} \leq C(p),
    \end{equation}
    where the constant $C(p)$ (in the case where $a \not = 1$) is given by
    \begin{align*}
        C(p) ={}& \left(\frac{b(a^{p+1} + a - 2)}{\sigma^2(a - 1)}\cdot \sqrt{\frac{a^{2p} - 1}{a^2 - 1}} + \frac{1+b}{b}\right)\left(\frac{b\sqrt{\left(a^{2p+2} - (p+1)a^2 + p\right)}}{\abs{a^2 - 1}} + 1\right) + \sqrt{\frac{a^{2p+2} - 1}{a^2 - 1}} - \frac{1}{b}.
    \end{align*}
    When $a = 1$, again by Lemma \ref{lemma:mat-norm-upper}, we can show that
    \begin{align*}
        \Vert S^\zeta \Vert \leq \left(\sum_{i=1}^p \sum_{j=1}^i \norm{\Phi(t+i, t+j)}^2\right)^{\frac{1}{2}} \leq \left(\sum_{i=1}^p \sum_{j=1}^i a^{2(i-j)}\right)^{\frac{1}{2}} = \sqrt{\frac{p(p+1)}{2}},
    \end{align*}
    and similarly,
    \[\norm{S^x} \leq \sqrt{p+1},~ \norm{M^\dagger} \leq \frac{b}{\sigma^2}\cdot \sqrt{p},~ \norm{S^v} \leq b\Vert S^\zeta \Vert,~ \Vert R^\zeta \Vert \leq \sqrt{p}.\]
    Therefore, when $a = 1$, the constant $C(p)$ should be set as follows to make inequality \eqref{lemma:implicit-switching-cost:e4} hold
    \[C(p) = \left(\frac{b \sqrt{p}}{\sigma^2}\left(\sqrt{p} + 2\right) + 1\right)\left(1 + b \sqrt{\frac{p(p+1)}{2}}\right) + \sqrt{p+1}\cdot\left(1 + \sqrt{\frac{p}{2}}\right).\]
    
    Since $G$ is convex and strongly smooth in $(x, u)$, and both $x, u$ are affine functions of $(y, z, r)$, we know $G(x(y, z, r), u(y, z, r))$ is convex and $\ell \cdot C(p)^2$-strongly smooth in $(y, z, r)$. Since $G(x, u)$ is $m_c$-strongly convex in $u$, by \eqref{lemma:implicit-switching-cost:e1}, we have
    \begin{align*}
        \nabla_{r}^2 G(x(y, z, w, r), u(y, z, w, r)) &\succeq V^\intercal \nabla_u^2 G(x, u) V\\
        &\succeq m_c I,
    \end{align*}
    where we use the fact that $\norm{V \nu}_2 = \norm{\nu}_2, \forall \nu \in \mathbb{R}^{mp-n}$ because the columns of $V$ are orthonormal in the last inequality. Therefore, by Lemma \ref{lemma:smoothness-opt-value}, we know that \eqref{lemma:implicit-switching-cost:e2} is convex and $L_2(p)$-strongly smooth in $(y, z)$, where
    \[L_2(p) := \ell \cdot C(p)^2 + \frac{\ell^2 \cdot C(p)^4}{m_c}.\]
    By Lemma \ref{lemma:smoothness-opt-value}, we also know that the optimal solution of \eqref{lemma:implicit-switching-cost:e2}
    \[r^*(x, z, \zeta) := \argmin_{r \in \mathbb{R}^{mp - n}}G(y(x, z, \zeta, r), v(x, z, \zeta, r))\]
    is $\ell \cdot C(p)^2/m_c$-Lipschitz. By \eqref{lemma:implicit-switching-cost:e3} and \eqref{lemma:implicit-switching-cost:e4}, we see that
    \[\psi_t^p(x, \zeta, z) = \begin{bmatrix}
            S^x - S^v M^\dagger \Phi(t+p, t) & S^v M^\dagger & S^\zeta - S^v M^\dagger R^\zeta & S^v V\\
            - M^\dagger \Phi(t+p, t) & M^\dagger & -M^\dagger R^\zeta & V
        \end{bmatrix} \cdot \begin{bmatrix}
            x\\
            z\\
            \zeta\\
            r^*(x, z, \zeta)
        \end{bmatrix}\]
    is $L_1(p)$-Lipschitz, where
    \[L_1(p) = C(p)(1 + \ell \cdot C(p)^2/m_c).\]


\section{Proof of Theorem \ref{thm:LTV-sensitivity}}\label{Appendix:thm:LTV-sensitivity}
The proof of Theorem \ref{thm:LTV-sensitivity} is based on the decision-point transformation introduced in Section \ref{sec:perturbation-ltv}.

Recall that $d$ denotes the controllability index, which has been defined in Definition \ref{def:controllability}. To show the perturbation bound of $\psi_{t}^{p}(\cdot,\cdot,\cdot)_{y_h}$, suppose $h$ and $p$ satisfy $q d \leq h < (q+1)d$ and $p = s d + r$, where $q,s,r \in \mathbb{N}$ and $0 \leq r < d$. Now we shall select the decision points as
\[y_0, y_{d}, \cdots, y_{(q-1)d}, y_h, y_{(q+2)d}, \cdots, y_{(s-1)d}, y_{\horizonvar},\]
which are also denoted by $y_{i_0}, \cdots, y_{i_{s-1}}$ for simplicity. Since the distance of any consecutive decision points falls in $[d, 2d)$, we can apply Lemma \ref{lemma:implicit-switching-cost} to bound the strong smoothness of switching costs. In the transformed SOCO problem, the disturbance input of the $(\tau-1)$-th time period is a vector $\bar{w}_{\tau-1} = \zeta_{i_{\tau-1}:i_{\tau}-1 } \in \mathbb{R}^{n \times (i_{\tau}-i_{\tau-1})}$. Each stage cost $\xi^{i_{\tau}-i_{\tau-1}}_{t}(x_{i_{\tau-1}}, \bar{w}_{\tau-1}, x_{i_{\tau}})$ is convex and $L_2(i_{\tau}-i_{\tau-1})$-strongly smooth by Lemma \ref{lemma:implicit-switching-cost}, and is thus $L_0$-strongly smooth by definition. Recall that the solution of the transformed SOCO problem is denoted by $\hat{\psi}(x_{t},\zeta,x_{t+p})$. Then by Theorem \ref{thm:SOCO-sensitivity} we have
\begin{align*}
    &\norm{\psi_{t}^p(x, \zeta, z)_{y_h} - \psi_{t}^p (x', \zeta', z')_{y_h}} \\
    {}={}& \norm{\hat{\psi}(x, \zeta, z)_{q} - \hat{\psi}(x', \zeta', z')_{q}} \\
    {}\leq{}& C_0 \left( \decayvar_0^{q-1} \norm{x-x'} + \sum_{\tau=0}^{s-2} \decayvar_0^{\abs{q-\tau}-1} \norm{\bar{w}_{\tau}-w'_{\tau}} + \decayvar_0^{(s-1)-q-1} \norm{z-z'} \right) \\
    {}={}& C_0 \left( \decayvar_0^{q-1} \norm{x-x'} + \sum_{\tau=0}^{s-2} \decayvar_0^{\abs{q-\tau}-1} \sum_{j=i_{\tau}}^{i_{\tau+1}-1} \norm{\zeta_j-\zeta'_j} + \decayvar_0^{(s-1)-q-1} \norm{z-z'} \right) \\
    {}\leq{}& \frac{C_0}{\decayvar_0} \left( \decayvar^{i_{q}-i_{0}} \norm{x-x'} + \sum_{\tau=0}^{s-2} \sum_{j=i_{\tau}}^{i_{\tau+1}-1} \decayvar^{\abs{j-i_{q}}} \norm{\zeta_j-\zeta'_j} + \decayvar^{i_{s-1}-i_{q}} \norm{z-z'} \right) \\
    {}={}& C \left(\decayvar^h \norm{x-x'} + \sum_{\tau = 0}^{p-1} \decayvar^{\abs{h-\tau}} \norm{\zeta_\tau-\zeta_\tau'} + \decayvar^{p-h} \norm{z-z'} \right).
\end{align*}
The last inequality holds because each interval is of length at most $(2d-1)$. Here the constants are
\begin{gather*}
    C_0 = \frac{2L_0}{m_c},~ \decayvar_0 = 1 - 2 \cdot \left( \sqrt{1 + (2L_0 / m_c)} + 1 \right)^{-1},\\
    C = C_0 / \decayvar_0 = \frac{2L_0}{m_c} \left( 1 - 2 \cdot \left( \sqrt{1 + (2L_0 / m_c)} + 1 \right)^{-1} \right)^{-1},~ \decayvar = \left(1 - 2\left(\sqrt{1 + (2L_0/m_c)} + 1\right)^{-1}\right)^{\frac{1}{2d - 1}}.
\end{gather*}

The proof of the perturbation bound of $\psi_{t}^{p}(\cdot,\cdot,\cdot)_{y_h}$ is quite similar. 
The only difference lies in the terminal cost, which can be addressed with the addition of a fixed auxiliary state. 
Specifically, we append $x_{\text{aux}} = 0$ to the end of the decision point sequence, and define a zero transition cost to the auxiliary state $\hat{c}_{s}(x_{t+p}, \bar{w}_{s-1}, x_{\text{aux}}) \equiv 0$ (note that $\hat{c}_{s}$ is trivially convex and $L_0$-strongly smooth). Denote the solution of the modified version of transformed SOCO problem by $\hat{\psi}'(x_t, \zeta, x_{\text{aux}})$, then by the same argument as above, we have
\begin{align*}
    \norm{\tilde{\psi}_{t}^p(x, \zeta; F)_{y_h} - \tilde{\psi}_{t}^p (x', \zeta'; F)_{y_h}} 
    &= \norm{\hat{\psi}'(x, \zeta, 0)_{q} - \hat{\psi}'(x', \zeta', 0)_{q}} \\
    &\leq \cdots
    \leq C \left(\decayvar^h \norm{x-x'} + \sum_{\tau = 0}^{p-1} \decayvar^{\abs{h-\tau}} \norm{\zeta_\tau-\zeta_\tau'} \right),
\end{align*}
where the constants are the same as previously defined. This finishes the proof of Theorem \ref{thm:LTV-sensitivity}.

\section{Stability of the Optimal Trajectory}\label{Appendix:coro:opt-stable}
In this section, we use the perturbation bound derived in Theorem \ref{thm:LTV-sensitivity} to establish the stability of the optimal trajectory, i.e., the optimal trajectory found by $\tilde{\psi}_t^p(x, \zeta; F)$ will stay in a ball centered at the origin if the disturbances are uniformly upper bounded. This result will be useful when showing the input-to-state stability and dynamic regret bound in Theorem \ref{thm:distance-MPC-infty}.

\begin{corollary}[Stability of the Optimal Trajectory]\label{coro:opt-stable}
For the predicted trajectory found by solving \eqref{opt:without-terminal} with prediction window $\horizonvar \ge d$, the norm of the $h$-th predictive state is bounded above by
\[\norm{\tilde{\psi}_{t}^\horizonvar(x, \zeta; F)_{y_h}} \leq C \left(\decayvar^h \norm{x} + \sum_{\tau = 0}^{\horizonvar-1}\decayvar^{\abs{h - \tau}} \norm{\zeta_\tau}\right) \leq C \decayvar^h \norm{x} + \frac{2 C}{1 - \decayvar} \sup_{\tau}\norm{\zeta_\tau},\]
where $C, \lambda$ are the same constants as in Theorem~\ref{thm:LTV-sensitivity}.
\end{corollary}

\begin{proof}[Proof of Corollary \ref{coro:opt-stable}]
Note that $\tilde{\psi}_t^p(0, 0; F)_{y_h} = 0$. By Theorem \ref{thm:LTV-sensitivity}, we see that
\begin{align*}
    \norm{\tilde{\psi}_{t}^p(x, \zeta; F)_{y_h}} &= \norm{\tilde{\psi}_{t}^p(x, \zeta; F)_{y_h} - \tilde{\psi}_{t}^p(0, 0; F)_{y_h}}\\
    &\leq C \left(\decayvar^h \norm{x} + \sum_{\tau = 0}^{p-1}\decayvar^{\abs{h - \tau}} \norm{\zeta_\tau}\right)\\
    &\leq C \decayvar^h \norm{x} + \frac{2 C}{1 - \decayvar} \sup_{\tau}\norm{\zeta_\tau},
\end{align*}
where the last inequality holds because
\[\sum_{\tau = 0}^{p-1}\decayvar^{\abs{h - \tau}} \leq \frac{2}{1 - \decayvar}.\]
\end{proof}

\section{Smoothness of the Optimal Cost}\label{Appendix:smooth-cost}
In this section, we use Lemma \ref{lemma:implicit-switching-cost} to show a result (Corollary \ref{coro:smooth-cost}) on the smoothness of the optimal cost of a $p$-step trajectory between the initial state $x$ and the terminal state $z$. Intuitively, Corollary \ref{coro:smooth-cost} implies that changing the initial/terminal state will not affect the optimal cost of a $p$-step trajectory between them significantly. This result will be useful when converting bounds on the distance between $PC_k$'s trajectory and the offline optimal trajectory to bounds on the cost $PC_k$ incurs.

\begin{corollary}\label{coro:smooth-cost}
For any time step $t$ and integer $p$ that satisfies $p \geq d$, function $\iota_t^p(\cdot, \zeta, \cdot)$ satisfies that
\[\iota_{t}^p(x, \zeta, z) \leq (1 + \eta) \iota_{t}^p\left(x', \zeta, z'\right) + \frac{L_0 + \ell_f}{2}\left(1 + \frac{1}{\eta}\right)\left(\norm{x' - x}^2 + \norm{z' - z}^2\right), \forall x, x', \zeta, z, z',\]
where $L_0$ is the same constant as in Theorem~\ref{thm:LTV-sensitivity}.
\end{corollary}

Before showing Corollary \ref{coro:smooth-cost}, we first show a property of strongly smooth functions for completeness.

\begin{lemma}\label{lemma:smooth-difference}
Suppose function $g: \mathbb{R}^n \to \mathbb{R}_+$ is convex, $\ell$-strongly smooth, and continuously differentiable. For all $x, y \in \mathbb{R}^n$ and $\eta > 0$, we have
\[g(x) \leq (1 + \eta) g(y) + \frac{\ell}{2}\left(1 + \frac{1}{\eta}\right)\norm{x - y}^2.\]
\end{lemma}
\begin{proof}[Proof of Lemma \ref{lemma:smooth-difference}]
\begin{align*}
    g(x) - g(y) &\leq \langle \nabla g(y), x - y\rangle + \frac{\ell}{2}\norm{x - y}^2\\
    &\leq \frac{\eta}{2\ell}\norm{\nabla g(y)}^2 + \frac{\ell}{2\eta}\norm{x - y}^2 + \frac{\ell}{2}\norm{x - y}^2\\
    &\leq \eta g(y) + \frac{\ell}{2}\left(1 + \frac{1}{\eta}\right)\norm{x - y}^2.
\end{align*}
where the second inequality follows from the generalized mean inequality and the last inequality holds because
\[ 0 \le g\left(y - \frac{\nabla g(y)}{\ell} \right) \le g(y) - \left\langle \nabla g(y), \frac{\nabla g(y)}{\ell} \right\rangle + \frac{\ell}{2}\norm{\frac{\nabla g(y)}{\ell}}^2 = g(y) - \frac{1}{2\ell} \norm{\nabla g(y)}^2\]
\end{proof}

Now we come back to the proof of Corollary \ref{coro:smooth-cost}.

\begin{proof}[Proof of Corollary \ref{coro:smooth-cost}]
When $d \leq p \leq 2d - 1$, since $\xi_t^p(x, \zeta, z)$ is $L_0$-strongly smooth by Lemma \ref{lemma:implicit-switching-cost}, we know
\[\iota_t^p(x, \zeta, z) = \xi_t^p(x, \zeta, z) + f_{t+p}(z)\]
is $(L_0 + \ell_f)$-strongly smooth. Therefore, by Lemma \ref{lemma:smooth-difference}, we obtain that
\[\iota_{t}^p(x, \zeta, z) \leq (1 + \eta) \iota_{t}^p\left(x', \zeta, z'\right) + \frac{L_0 + \ell_f}{2}\left(1 + \frac{1}{\eta}\right)\left(\norm{x' - x}^2 + \norm{z' - z}^2\right).\]

When $p = 2d$, let $x_1 := \psi_t^p(x, \zeta, z)_{y_d}$, and we obtain that
\begin{align*}
    \iota_t^p(x, \zeta, z) ={}& \iota_t^{d}(x, \zeta_{0:d-1}, x_1) + \iota_{t+d}^{d}(x_1, \zeta_{d:2d-1}, z)\\
    \leq{}& (1 + \eta) \iota_t^{d}(x', \zeta_{0:d-1}, x_1) + \frac{L_0 + \ell_f}{2}\left(1 + \frac{1}{\eta}\right)\norm{x - x'}^2\\
    &+ (1 + \eta) \iota_{t+d}^{d}(x_1, \zeta_{d:2d-1}, z') + \frac{L_0 + \ell_f}{2}\left(1 + \frac{1}{\eta}\right)\norm{z - z'}^2\\
    \leq{}& (1 + \eta) \iota_{t}^p\left(x', \zeta, z'\right) + \frac{L_0 + \ell_f}{2}\left(1 + \frac{1}{\eta}\right)\left(\norm{x' - x}^2 + \norm{z' - z}^2\right).
\end{align*}

When $p > 2d$, let $x_1 := \psi_t^p(x, \zeta, z)_{y_d}, x_2 := \psi_t^p(x, \zeta, z)_{y_{p-d}}$, and we obtain that
\begin{align*}
    \iota_t^p(x, \zeta, z) ={}& \iota_t^{d}(x, \zeta_{0:d-1}, x_1) +\iota_{t+d}^{p-2d}(x_1, \zeta_{d:p-d-1}, x_2) + \iota_{t+p-d}^{d}(x_2, \zeta_{p-d:p-1}, z)\\
    \leq{}& (1 + \eta) \iota_t^{d}(x', \zeta_{0:d-1}, x_1) + \frac{L_0 + \ell_f}{2}\left(1 + \frac{1}{\eta}\right)\norm{x - x'}^2\\
    &+ \iota_{t+d}^{p-2d}(x_1, \zeta_{d:p-d-1}, x_2)\\
    &+ (1 + \eta) \iota_{t+p-d}^{d}(x_2, \zeta_{p-d:p-1}, z') + \frac{L_0 + \ell_f}{2}\left(1 + \frac{1}{\eta}\right)\norm{z - z'}^2\\
    \leq{}& (1 + \eta) \iota_{t}^p\left(x', \zeta, z'\right) + \frac{L_0 + \ell_f}{2}\left(1 + \frac{1}{\eta}\right)\left(\norm{x' - x}^2 + \norm{z' - z}^2\right).
\end{align*}
\end{proof}

\section{Proof of Lemma \ref{lemma:one-step-diff}}\label{Appendix:lemma:one-step-diff}
For simplicity, we will use the shorthand notations
\[\tilde{\psi}_t^p(x; F) := \tilde{\psi}_t^p(x, w_{t:t+p-1}; F) \text{ and }\psi_t^p(x, z) := \psi_t^p(x, w_{t:t+p-1}, z)\]
throughout the proof, since the indices of the disturbances can be inferred from the starting time $t$ and horizon $p$. We also define
    \[z := \tilde{\psi}_{t}^p\left(x_t; F\right)_{y_p}, z' := \tilde{\psi}_{t}^{p+1}\left(x_t; F\right)_{y_p}.\]
Then it is straightforward to see that
    \begin{subequations}\label{equ:one-step-diff}
    \begin{align}
        \norm{\tilde{\psi}_{t}^p\left(x_t; F\right)_{y_h} - \tilde{\psi}_{t}^{p+1}\left(x_t; F\right)_{y_h}}
        ={}&\norm{\psi_{t}^p\left(x_t, z\right)_{y_h} - \psi_{t}^p\left(x_t, z'\right)_{y_h}}\label{equ:one-step-diff:s1}\\
        \leq{}& C \decayvar^{p-h}\norm{z - z'}\label{equ:one-step-diff:s2}\\
        \leq{}& 2 C \decayvar^{p-h}\left(C\decayvar^p \norm{x_t} + \frac{2 C}{1 - \decayvar} D\right)\label{equ:one-step-diff:s3}.
    \end{align}
    \end{subequations}
where we use the definition of $\psi$ and $\tilde{\psi}$ in \eqref{equ:one-step-diff:s1}, Theorem \ref{thm:LTV-sensitivity} in \eqref{equ:one-step-diff:s2}, and Corollary \ref{coro:opt-stable} in \eqref{equ:one-step-diff:s3}.

\section{Proof of Theorem \ref{thm:distance-MPC-infty}}\label{Appendix:thm:distance-MPC-infty}
Throughout the proof, we will use $\{(\hat{x}_t, \hat{u}_t)\}$ to denote the trajectory of predictive control with prediction window $T$ ($PC_T$). Recall that $\{(x_t, u_t)\}$ denotes the trajectory of predictive control with prediction window $k$ ($PC_k$), and $\{(x_t^*, u_t^*)\}$ denotes the offline optimal trajectory ($OPT$), i.e., the optimal solution of \eqref{equ:online_control_problem}.

For simplicity, we will use the shorthand notations
\[\tilde{\psi}_t^p(x; F) := \tilde{\psi}_t^p(x, w_{t:t+p-1}; F) \text{ and }\psi_t^p(x, z) := \psi_t^p(x, w_{t:t+p-1}, z)\]
throughout the proof.

Since $x_{t+1} = \tilde{\psi}_t^p(x_t; F)_{y_1}$, for all $2 \leq i \leq k$, we have
\begin{equation}\label{thm:ISS-e0}
    \tilde{\psi}_{t-i}^k (x_{t-i}; F)_{y_i} = \tilde{\psi}_{t-i+1}^{k-1}(x_{t-i+1}; F)_{y_{i-1}}.
\end{equation}
Therefore, it can be shown that, for $k \leq t \leq T - k$,
\begin{subequations}\label{thm:ISS-e1}
\begin{align}
    \norm{x_t} ={}& \norm{\tilde{\psi}_{t-1}^k(x_{t-1}; F)_{y_1}}\nonumber\\
    \leq{}& \sum_{i = 1}^{k-1} \norm{\tilde{\psi}_{t-i}^k(x_{t-i}; F)_{y_i} - \tilde{\psi}_{t-i-1}^k(x_{t-i-1}; F)_{y_{i+1}}} + \norm{\tilde{\psi}_{t-k}^k(x_{t-k}; F)_{y_k}}\label{thm:ISS-e1:s1}\\
    ={}& \sum_{i = 1}^{k-1} \norm{\tilde{\psi}_{t-i}^k(x_{t-i}; F)_{y_i} - \tilde{\psi}_{t-i}^{k-1}(x_{t-i}; F)_{y_{i}}} + \norm{\tilde{\psi}_{t-k}^k(x_{t-k}; F)_{y_k}}\label{thm:ISS-e1:s2}\\
    \leq{}& \sum_{i=1}^{k-1} 2 C \decayvar^{k - 1 - i}\left(C \decayvar^{k-1} \norm{x_{t-i}} + \frac{2C}{1 - \decayvar} D\right) + \left(C \decayvar^k \norm{x_{t-k}} + \frac{2C}{1 - \decayvar} D\right)\label{thm:ISS-e1:s3}\\
    \leq{}& C \decayvar^{k-1} \left(\decayvar \norm{x_{t-k}} + 2 C\sum_{i=1}^{k-1}\decayvar^{k-1-i}\norm{x_{t-i}}\right) + \frac{2 C}{1 - \decayvar}\left(1 + \frac{2 C}{1 - \decayvar}\right) D,\nonumber
\end{align}
\end{subequations}
where we use triangle inequality in \eqref{thm:ISS-e1:s1}, \eqref{thm:ISS-e0} in \eqref{thm:ISS-e1:s2}, and Lemma \ref{lemma:one-step-diff} and Corollary \ref{coro:opt-stable} in \eqref{thm:ISS-e1:s3}.

By a similar argument, for $1 \leq t \leq k$, we have
\begin{align}\label{thm:ISS-e2}
    \norm{x_t} ={}& \norm{\tilde{\psi}_{t-1}^k(x_{t-1}; F)_{y_1}}\nonumber\\
    \leq{}& \sum_{i = 1}^{t-1} \norm{\tilde{\psi}_{t-i}^k(x_{t-i}; F)_{y_i} - \tilde{\psi}_{t-i-1}^k(x_{t-i-1}; F)_{y_{i+1}}} + \norm{\tilde{\psi}_{0}^k(x_{0}; F)_{y_t}}\nonumber\\
    ={}& \sum_{i = 1}^{t-1} \norm{\tilde{\psi}_{t-i}^k(x_{t-i}; F)_{y_i} - \tilde{\psi}_{t-i}^{k-1}(x_{t-i}; F)_{y_{i}}} + \norm{\tilde{\psi}_{0}^k(x_{0}, F)_{y_t}}\nonumber\\
    \leq{}& \sum_{i=1}^{t-1} 2 C \decayvar^{k - 1 - i}\left(C \decayvar^{k-1} \norm{x_{t-i}} + \frac{2C}{1 - \decayvar} D\right) + \left(C \decayvar^t \norm{x_{0}} + \frac{2C}{1 - \decayvar} D\right)\nonumber\\
    \leq{}& C \decayvar^{k-1} \left(2 C\sum_{i=1}^{t-1}\decayvar^{k-1-i}\norm{x_{t-i}}\right) + C\norm{x_0} + \frac{2 C}{1 - \decayvar}\left(1 + \frac{2 C}{1 - \decayvar}\right) D,
\end{align}

Recall that, under the assumption of \eqref{equ:regret-condition}, the sum of coefficients in \eqref{thm:ISS-e1} and \eqref{thm:ISS-e2} are upper bounded by
\[C \decayvar^{k-1} \left(2 C\sum_{i=1}^{t-1}\decayvar^{k-1-i}\right) \leq 1 - \delta, \; C \decayvar^{k-1} \left(\decayvar + 2 C\sum_{i=1}^{k-1}\decayvar^{k-1-i}\right) < 1 - \delta.\]
Then, using inequalities \eqref{thm:ISS-e1} and \eqref{thm:ISS-e2}, we can show by induction that, for $t \leq T - k$,
\begin{equation}\label{thm:ISS-e3}
    \norm{x_t} \leq \frac{C}{\delta}\cdot (1 - \delta)^{\max(0, t - k)} \norm{x_0} + \frac{2 C}{\delta (1 - \decayvar)}\left(1 + \frac{2 C}{1 - \decayvar}\right) D.
\end{equation}
For $t \geq T - k + 1$, by Corollary \ref{coro:opt-stable}, we see that
\begin{align*}
    \norm{x_t} &= \norm{\tilde{\psi}_{T-k}^k(x_{T-k}, w_{T-k:T-1}; 0)_{y_{t + k - T}}}\\
    &\leq C \decayvar^{t+k-T}\norm{x_{T-k}} + \frac{2C D}{1 - \decayvar}\\
    &\leq \frac{C^2}{\delta}\cdot (1 - \delta)^{T - 2k} \decayvar^{t+k-T}\norm{x_0} + \left(\frac{2 C^2}{\delta (1 - \decayvar)}\left(1 + \frac{2 C}{1 - \decayvar}\right) + \frac{2C}{1 - \decayvar}\right)D.
\end{align*}

This finishes the proof of ISS of $PC_k$.

By Lemma \ref{lemma:one-step-diff} and \eqref{thm:ISS-e3}, we also see that for $t \leq T - k$,
\begin{align}\label{thm:distance-MPC-infty:e1}
    \norm{\tilde{\psi}_{t}^k\left(x_t; F\right)_{y_1} - \tilde{\psi}_{t}^{T-t}\left(x_t; F\right)_{y_1}} &\leq \sum_{p=k}^{T - t}\norm{\tilde{\psi}_{t}^p\left(x_t; F\right)_{y_1} - \tilde{\psi}_{t}^{p+1}\left(x_t; F\right)_{y_1}}\nonumber\\
    &\leq \sum_{p=k}^\infty 2 C \decayvar^{p-1}\left(C\decayvar^p \norm{x_t} + \frac{2 C}{1 - \decayvar} D\right)\nonumber\\
    &= \frac{2 C^2}{\decayvar (1 - \decayvar^2)}\cdot \decayvar^{2k} \norm{x_t} + \frac{4 C^2}{\decayvar (1 - \decayvar)^2}\cdot \decayvar^k D\nonumber\\
    &= O\left(\left(D + \frac{\decayvar^k(\norm{x_0} + D)}{\delta}\right)\decayvar^k\right).
\end{align}
We further obtain that for $t \leq T - k$,
\begin{subequations}\label{thm:distance-MPC-infty:e2}
\begin{align}
    \norm{x_t - \hat{x}_t}
    ={}& \norm{x_t - \tilde{\psi}_0^T(x_0; F)_{y_t}}\nonumber{}\\
    \leq{}& \norm{x_t - \tilde{\psi}_{t-1}^{T-t+1}(x_{t-1}; F)_{y_1}} + \sum_{i=1}^{t-1} \norm{\tilde{\psi}_{t-i}^{T - t + i}(x_{t-i}; F)_{y_i} - \tilde{\psi}_{t-i-1}^{T - t + i + 1}(x_{t-i-1}; F)_{y_{i+1}}}\nonumber\\
    \leq{}& \norm{x_t - \tilde{\psi}_{t-1}^{T - t + 1}(x_{t-1}; F)_{y_1}} + \sum_{i = 1}^{t-1} C \decayvar^{i} \norm{x_{t-i} - \tilde{\psi}_{t-i-1}^{T - t + i + 1}(x_{t-i-1}; F)_{y_1}}\label{thm:distance-MPC-infty:e2:s1}\\
    ={}& O\left(\left(D + \frac{\decayvar^k(\norm{x_0} + D)}{\delta}\right)\decayvar^k\right),\label{thm:distance-MPC-infty:e2:s2}
\end{align}
\end{subequations}
where in \eqref{thm:distance-MPC-infty:e2:s1}, we use Theorem \ref{thm:LTV-sensitivity} and the fact that $\tilde{\psi}_{t-i-1}^{T-t+i+1}(x_{t-i-1})_{y_{i+1}}$ can be written as
\[\tilde{\psi}_{t-i-1}^{T - t + i + 1}(x_{t-i-1}; F)_{y_{i+1}} = \tilde{\psi}_{t-i}^{T - t + i}\left(\tilde{\psi}_{t-i-1}^{T - t + i + 1}(x_{t-i-1}; F)_{y_1}; F\right)_{y_i};\]
in \eqref{thm:distance-MPC-infty:e2:s2}, we use \eqref{thm:distance-MPC-infty:e1} and the following observations
\begin{gather*}
    \norm{x_{t-i} - \tilde{\psi}_{t-i-1}^{T - t + i + 1}(x_{t-i-1}; F)_{y_1}} = \norm{\tilde{\psi}_{t-i-1}^k\left(x_{t-i-1}; F\right)_{y_1} - \tilde{\psi}_{t-i-1}^{T - t + i + 1}(x_{t-i-1}; F)_{y_1}},\\
    1 + \sum_{i=1}^{t-1} C \decayvar^i \leq 1 + \frac{C}{1 - \decayvar} = O(1).
\end{gather*}

By Corollary \ref{coro:opt-stable} and triangle inequality, we see that
\[\norm{x_T^* - \hat{x}_T} \leq 2C\decayvar^T \norm{x_0} + \frac{4C D}{1 - \decayvar}.\]
Then by Theorem \ref{thm:LTV-sensitivity}, the following holds for all $t \leq T - k$:
\[\norm{x_t^* - \hat{x}_t} = \norm{\psi_0^T(x_0, x^*_T) - \psi_0^T(x_0, \hat{x}_T)} \leq C \decayvar^k \left(2C\decayvar^T \norm{x_0} + \frac{4C D}{1 - \decayvar}\right).\]
Combining this inequality with \eqref{thm:distance-MPC-infty:e2} gives
\begin{equation}\label{thm:distance-MPC-infty:e2-0}
    \norm{x_t - x_t^*} = O\left(\left(D + \frac{\decayvar^k(\norm{x_0} + D)}{\delta}\right)\decayvar^k\right), \; \forall t \leq T - k.
\end{equation}

Since
\[(u_t - {u}_t^*) = B_t^\dagger \left((x_{t+1} - {x}_{t+1}^*) - A_t (x_t - {x}_t^*)\right),\]
we have
\[\norm{u_t - {u}_t^*} \leq b'\left(\norm{x_{t+1} - {x}_{t+1}^*} + a \norm{x_t - {x}_t^*}\right).\]
Therefore, by Corollary \ref{coro:smooth-cost}, for any $\eta > 0$ we have
\begin{align}\label{thm:distance-MPC-infty:e3}
    &\iota_{t}^1(x_{t}, x_{t+1}) - (1 + \eta)\iota_{t}^1({x}_{t}^*, {x}_{t+1}^*)\nonumber\\
    ={}& \left(f_{t+1}(x_{t+1}) - (1 + \eta) f_{t+1}({x}_{t+1}^*)\right) + \left(c_{t+1}(u_{t}) - (1 + \eta) c_{t+1}({u}_{t}^*)\right)\nonumber\\
    \leq{}& \frac{1}{2}\left(1 + \frac{1}{\eta}\right)\left(\ell_f \norm{x_{t+1} - {x}_{t+1}^*}^2 + \ell_c \norm{u_t - {u}_t^*}^2\right)\nonumber\\
    \leq{}& \frac{1}{2}\left(1 + \frac{1}{\eta}\right) \left(\ell_f + 2(b')^2 \ell_c\right)\norm{x_{t+1} - {x}_{t+1}^*}^2 + \frac{1}{2}\left(1 + \frac{1}{\eta}\right) 2 a^2 (b')^2 \ell_c \norm{x_t - {x}_t^*}^2\nonumber\\
    \leq{}& \left(1 + \frac{1}{\eta}\right)\cdot \frac{L_4}{2} \left(\norm{x_{t} - {x}_{t}^*}^2 + \norm{x_{t+1} - {x}_{t+1}^*}^2\right),
\end{align}
where
\[L_4 := \ell_f + 2(b')^2 \ell_c + 2 a^2 (b')^2 \ell_c.\]

Then, for any $\eta > 0$, we obtain the following inequality:
\begin{subequations}\label{equ:regret_e1}
\begin{align}
    &cost(PC_k) - (1 + \eta) cost(OPT)\nonumber\\
    ={}& \left(\sum_{t = 0}^{T-k-1} \iota_{t}^1(x_{t}, x_{t + 1}) + \iota_{T-k}^k(x_{T-k}, x_T)\right) - (1 + \eta)\left(\sum_{t = 0}^{T-k-1} \iota_{t}^1({x}_{t}^*, {x}_{t + 1}^*) + \iota_{T-k}^k(x_{T-k}^*, x_T^*)\right)\nonumber\\
    ={}& \sum_{t = 0}^{T-k-1} \left(\iota_{t}^1(x_{t}, x_{t+1}) - (1 + \eta)\iota_{t}^1(x_{t}^*, x_{t+1}^*)\right) + \left(\iota_{T-k}^k(x_{T-k}, x_T) - (1 + \eta)\iota_{T-k}^k(x_{T-k}^*, x_T^*)\right) \nonumber\\
    \leq{}& \sum_{t = 0}^{T-k-1} \left(\iota_{t}^1(x_{t}, x_{t+1}) - (1 + \eta)\iota_{t}^1(x_{t}^*, x_{t+1}^*)\right) + \left(\iota_{T-k}^k(x_{T-k}, x_T^*) - (1 + \eta)\iota_{T-k}^k(x_{T-k}^*, x_T^*)\right) \label{equ:regret_e1:s1}\\
    \leq{}& \left(1 + \frac{1}{\eta}\right)\cdot \frac{L_4}{2} \sum_{t = 0}^{T - k - 1} \left(\norm{x_{t} - x_{t}^*}^2 + \norm{x_{t+1} - x_{t+1}^*}^2\right) + \left(1 + \frac{1}{\eta}\right)\cdot \frac{L_0 + \ell_f}{2}\norm{x_{T-k} - x_{T-k}^*}^2 \label{equ:regret_e1:s2}\\
    ={}& \left(1 + \frac{1}{\eta}\right)\cdot L_4 \sum_{t = 0}^{T - k - 1} \norm{x_{t} - x_{t}^*}^2 + \left(1 + \frac{1}{\eta}\right)\cdot \frac{L_4 + L_0 + \ell_f}{2}\norm{x_{T-k} - x_{T-k}^*}^2\nonumber\\ 
    \leq{}& \left(1 + \frac{1}{\eta}\right) O\left(\left(D + \frac{\decayvar^k(\norm{x_0} + D)}{\delta}\right)^2 \decayvar^{k} T\right), \label{equ:regret_e1:s3}
\end{align}
\end{subequations}
where we use the fact that our algorithm $PC_k$ plans optimally after time step $T-k$ in \eqref{equ:regret_e1:s1}; we also use \eqref{thm:distance-MPC-infty:e3} and Corollary \ref{coro:smooth-cost} in \eqref{equ:regret_e1:s2}, and \eqref{thm:distance-MPC-infty:e2} in \eqref{equ:regret_e1:s3}.


To bound the optimal cost, we consider a suboptimal controller inspired by the decision-point transformation, where the controller forces the states $x_d, x_{2d}, \cdots, x_{(v-1)d}, x_{vd+r}$ to be $0$ ($d$ is the controllability index, and $T = vd+r$). The cost of this suboptimal control is determined by the transformed transition cost $\xi_{t}^{p}(\cdot, \cdot, \cdot)$ between each pair of consecutive decision points. By strong smoothness of $\xi_t^p(\cdot, \cdot, \cdot)$ proven in Lemma \ref{lemma:implicit-switching-cost}, we have
\[\xi_t^p(x, \zeta, 0) \leq \frac{1}{2} L_2(p) \left( \norm{\zeta}^2 + \norm{x}^2 \right) \leq \frac{L_0 D^2}{2} p + \frac{L_0}{2} \norm{x}^2,\]
where $L_0 = \max_{d \leq p \leq 2d-1} L_2(p)$. These inequalities add up to
\begin{align*}
    cost(OPT) 
    &\leq \xi_{0}^{d}(x_0,w_{0:d-1},0) + \sum_{\tau=1}^{v-2} \xi_{\tau d}^{d}(0,w_{\tau d:(\tau+1)d-1},0) + \xi_{(t-1)d}^{d+r}(0,w_{(v-1)d:T-1},0)\\
    &\leq \frac{L_0 D^2}{2} T + \frac{L_0}{2} \norm{x_0}^2\\
    &= O(D^2 T + \norm{x_0}^2).
\end{align*}
Hence $cost(OPT) = O(D^2 T + \norm{x_0}^2)$. Now we can take $\eta = \Theta(\lambda^k)$ in \eqref{equ:regret_e1} to get a regret bound of
\[cost(PC_k) - cost(OPT) = O\left(\left(D + \frac{\decayvar^k(\norm{x_0} + D)}{\delta}\right)^2 \decayvar^{k} T + \lambda^k \norm{x_0}^2\right).\]


\section{Proof of Theorem~\ref{thm:new-competitive-ratio}}\label{Appendix:thm:new-competitive-ratio}
To show Theorem \ref{thm:new-competitive-ratio}, we first state a result (Lemma \ref{lemma:new-one-step-diff}) that bounds the change in decision points against the change in prediction window. It is similar with Lemma \ref{lemma:one-step-diff}, but the key difference is that the right hand side of the inequality in Lemma \ref{lemma:new-one-step-diff} is composed by offline optimal states and their distance to the algorithm's states. This result is critical for showing the competitive ratio, and relies on the assumption that $F$ is the indicator function of the origin.

\begin{lemma}\label{lemma:new-one-step-diff}
    Under the same assumptions as Theorem \ref{thm:new-competitive-ratio}, for any integers $\horizonvar, h$ such that $\horizonvar \geq h \geq 1$ and time step $t < T - \horizonvar$, we have
    \[\norm{\tilde{\psi}_{t}^\horizonvar \left(x_t, w_{t:t+\horizonvar-1}; F\right)_{y_h} - \tilde{\psi}_{t}^{\horizonvar +1}\left(x_t, w_{t:t+\horizonvar}; F\right)_{y_h}} \leq C \lambda^{p-h} \left(\norm{x_{t+p}^*} + C \lambda^p\norm{x_t - x_t^*} + C \norm{x_{t+p+1}^*}\right).\]
\end{lemma}

\begin{proof}[Proof of Lemma \ref{lemma:new-one-step-diff}]
For simplicity, we will use the shorthand notations
\[\tilde{\psi}_t^p(x; F) := \tilde{\psi}_t^p(x, w_{t:t+p-1}; F) \text{ and }\psi_t^p(x, z) := \psi_t^p(x, w_{t:t+p-1}, z)\]
throughout the proof, since the indices of the disturbances can be inferred from the starting time $t$ and horizon $p$. We also define
    \[z := \tilde{\psi}_{t}^p\left(x_t; F\right)_{y_p}, z' := \tilde{\psi}_{t}^{p+1}\left(x_t; F\right)_{y_p}.\]
Note that $z = 0$. Then it is straightforward to see that
    \begin{subequations}\label{equ:new-one-step-diff}
    \begin{align}
        \norm{\tilde{\psi}_{t}^p\left(x_t; F\right)_{y_h} - \tilde{\psi}_{t}^{p+1}\left(x_t; F\right)_{y_h}}
        ={}&\norm{\psi_{t}^p\left(x_t, z\right)_{y_h} - \psi_{t}^p\left(x_t, z'\right)_{y_h}}\label{equ:new-one-step-diff:s1}\\
        \leq{}& C \decayvar^{p-h}\norm{z - z'}\label{equ:new-one-step-diff:s2}\\
        \leq{}& C \decayvar^{p-h}\left(\norm{z - x_{t+p}^*} + \norm{z' - x_{t+p}^*}\right)\nonumber\\
        \leq{}& C \decayvar^{p-h}\left(\norm{x_{t+p}^*} + \norm{\psi_t^{p+1}(x_t, 0)_{y_p} - \psi_t^{p+1}(x_t^*, x_{t+p+1}^*)_{y_p}}\right)\nonumber\\
        \leq{}& C \lambda^{p-h} \left(\norm{x_{t+p}^*} + C \lambda^p\norm{x_t - x_t^*} + C \norm{x_{t+p+1}^*}\right)\label{equ:new-one-step-diff:s3},
    \end{align}
    \end{subequations}
where we use the definition of $\psi$ and $\tilde{\psi}$ in \eqref{equ:new-one-step-diff:s1}, Theorem \ref{thm:LTV-sensitivity} in \eqref{equ:new-one-step-diff:s2} and \eqref{equ:new-one-step-diff:s3}.
\end{proof}

Now we come back to the proof of Theorem \ref{thm:new-competitive-ratio}. 

\begin{proof}[Proof of Theorem \ref{thm:new-competitive-ratio}]
Throughout the proof, we will use $\{(\hat{x}_t, \hat{u}_t)\}$ to denote the trajectory of predictive control with prediction window $T$ ($PC_T$). Recall that $\{(x_t, u_t)\}$ denotes the trajectory of predictive control with prediction window $k$ ($PC_k$), and $\{(x_t^*, u_t^*)\}$ denotes the offline optimal trajectory ($OPT$), i.e., the optimal solution of \eqref{equ:online_control_problem}.

By Lemma \ref{lemma:new-one-step-diff}, we see that for $t \leq T - k$,
\begin{align}\label{thm:new-distance-MPC-infty:e1}
    \norm{\tilde{\psi}_{t}^k\left(x_t; F\right)_{y_1} - \tilde{\psi}_{t}^{T-t}\left(x_t; F\right)_{y_1}} &\leq \sum_{p=k}^{T - t - 1}\norm{\tilde{\psi}_{t}^p\left(x_t; F\right)_{y_1} - \tilde{\psi}_{t}^{p+1}\left(x_t; F\right)_{y_1}}\nonumber\\
    &\leq \sum_{p=k}^{T-t-1} C \lambda^{p-1} \left(\norm{x_{t+p}^*} + C \lambda^p\norm{x_t - x_t^*} + C \norm{x_{t+p+1}^*}\right)\nonumber\\
    &\leq \frac{C^2}{\decayvar(1 - \decayvar^2)}\cdot \decayvar^{2k}\norm{x_t - x_t^*} + C(C + 1)\sum_{p=k}^{T - t} \lambda^{p-2} \norm{x_{t+p}^*}.
\end{align}
We further obtain that for $t \leq T - k$,
\begin{subequations}\label{thm:new-distance-MPC-infty:e2}
\begin{align}
    \norm{x_t - \hat{x}_t}
    ={}& \norm{x_t - \tilde{\psi}_0^T(x_0; F)_{y_t}}\nonumber{}\\
    \leq{}& \norm{x_t - \tilde{\psi}_{t-1}^{T-t+1}(x_{t-1}; F)_{y_1}} + \sum_{i=1}^{t-1} \norm{\tilde{\psi}_{t-i}^{T - t + i}(x_{t-i}; F)_{y_i} - \tilde{\psi}_{t-i-1}^{T - t + i + 1}(x_{t-i-1}; F)_{y_{i+1}}}\nonumber\\
    \leq{}& \norm{x_t - \tilde{\psi}_{t-1}^{T - t + 1}(x_{t-1}; F)_{y_1}} + \sum_{i = 1}^{t-1} C \decayvar^{i} \norm{x_{t-i} - \tilde{\psi}_{t-i-1}^{T - t + i + 1}(x_{t-i-1}; F)_{y_1}}\label{thm:new-distance-MPC-infty:e2:s1}\\
    \leq{}& \sum_{i = 0}^{t-1} C \decayvar^{i} \norm{x_{t-i} - \tilde{\psi}_{t-i-1}^{T - t + i + 1}(x_{t-i-1}; F)_{y_1}}\label{thm:new-distance-MPC-infty:e2:s1-1}\\
    \leq{}& \frac{C^3}{\decayvar(1 - \decayvar^2)}\cdot \decayvar^{2k} \left(\sum_{i=1}^{t-1}\lambda^{i-1} \norm{x_{t-i} - x_{t-i}^*}\right) + C^2(C + 1)\sum_{i=0}^{t-1} \sum_{p = k}^{T - t + i + 1} \lambda^{i + p - 2}\norm{x_{t-i-1+p}^*}\label{thm:new-distance-MPC-infty:e2:s1-2}\\
    \leq{}& \frac{C^3}{\decayvar(1 - \decayvar^2)}\cdot \decayvar^{2k} \left(\sum_{i=1}^{t-1}\lambda^{i-1} \norm{x_{t-i} - x_{t-i}^*}\right) + \frac{C^2 (C + 1)}{\lambda^2(1 - \lambda^2)}\cdot \lambda^k \sum_{j=1}^T \lambda^{\abs{j - t - k + 1}} \norm{x_j^*},\label{thm:new-distance-MPC-infty:e2:s2}
\end{align}
\end{subequations}
where in \eqref{thm:new-distance-MPC-infty:e2:s1}, we use Theorem \ref{thm:LTV-sensitivity} and the fact that $\tilde{\psi}_{t-i-1}^{T-t+i+1}(x_{t-i-1})_{y_{i+1}}$ can be written as
\[\tilde{\psi}_{t-i-1}^{T - t + i + 1}(x_{t-i-1}; F)_{y_{i+1}} = \tilde{\psi}_{t-i}^{T - t + i}\left(\tilde{\psi}_{t-i-1}^{T - t + i + 1}(x_{t-i-1}; F)_{y_1}; F\right)_{y_i}.\]
In \eqref{thm:new-distance-MPC-infty:e2:s1-1}, we use $C \geq 1$. In \eqref{thm:new-distance-MPC-infty:e2:s1-2}, we use \eqref{thm:new-distance-MPC-infty:e1} and the following observation
\begin{gather*}
    \norm{x_{t-i} - \tilde{\psi}_{t-i-1}^{T - t + i + 1}(x_{t-i-1}; F)_{y_1}} = \norm{\tilde{\psi}_{t-i-1}^k\left(x_{t-i-1}; F\right)_{y_1} - \tilde{\psi}_{t-i-1}^{T - t + i + 1}(x_{t-i-1}; F)_{y_1}}.
\end{gather*}
We reorganize the second term in \eqref{thm:new-distance-MPC-infty:e2:s1-2} to obtain \eqref{thm:new-distance-MPC-infty:e2:s2}. Specifically, for each outer loop index $i$, we see the coefficients before $\norm{x_j^*}$ (ignore $C^2(C + 1)$) can be written according to the table below:

\begin{center}
\begin{tabular}{|c||c|c|c|c|c|c|c|}
\hline
index $j$  & $\cdots$ & $t+k-3$ & $t+k-2$ & $t+k-1$ & $t+k$ & $\cdots$ & $T$ \\
\hline
\hline
 $i = 0$ & $\cdots$ & $0$ & $0$ & $\lambda^{k-2}$ & $\lambda^{k-1}$ & $\cdots$ & $\lambda^{T-t-1}$\\
\hline
$i = 1$ & $\cdots$ & $0$ & $\lambda^{k-1}$ & $\lambda^{k}$ & $\lambda^{k+1}$ & $\cdots$ & $\lambda^{T-t+1}$\\
\hline
$i = 2$ & $\cdots$ & $\lambda^{k}$ & $\lambda^{k+1}$ & $\lambda^{k+2}$ & $\lambda^{k+3}$ & $\cdots$ & $\lambda^{T-t+3}$\\
\hline
$\vdots$ &  & $\vdots$ & $\vdots$ & $\vdots$ & $\vdots$ & & $\vdots$\\
\hline
\end{tabular}
\end{center}

Note that by Theorem \ref{thm:LTV-sensitivity}, the following holds for all $t \leq T - k$:
\[\norm{x_t^* - \hat{x}_t} = \norm{\psi_0^T(x_0, x^*_T)_{y_t} - \psi_0^T(x_0, 0)_{y_t}} \leq C \decayvar^{T - t} \norm{x_T^*}.\]
Combining this inequality with \eqref{thm:new-distance-MPC-infty:e2} gives that for all $t \leq T - k$,
\begin{align}\label{thm:new-distance-MPC-infty:e3}
    \norm{x_t - x_t^*} \leq{}& \frac{C^3}{\decayvar(1 - \decayvar^2)}\cdot \decayvar^{2k} \left(\sum_{i=1}^{t-1}\lambda^{i-1} \norm{x_{t-i} - x_{t-i}^*}\right)\nonumber\\
    &+ \frac{C^2 (C + 1)}{\lambda^2(1 - \lambda^2)}\cdot \lambda^k \sum_{j=1}^T \lambda^{\abs{j - t - k +1}}\norm{x_j^*} + C \decayvar^{T - t} \norm{x_T^*}.
\end{align}
Note that for any $T$ nonnegative real numbers $r_1, r_2, \ldots, r_T$ and any $\tau \in [1, T]$, we have the inequality
\[\left(\sum_{i=1}^T \lambda^{\abs{i - \tau}} r_i\right)^2 \leq \left(\sum_{i=1}^T \lambda^{\abs{i - \tau}}\right)\left(\sum_{i=1}^T \lambda^{\abs{i - \tau}} r_i^2\right) \leq \frac{2}{1 - \lambda}\cdot \left(\sum_{i=1}^T \lambda^{\abs{i - \tau}} r_i^2\right).\]
Taking square of both sides of \eqref{thm:new-distance-MPC-infty:e3} and apply the above inequality gives for all $t \leq T - k$,
\begin{subequations}\label{thm:new-distance-MPC-infty:e4}
\begin{align}
    \norm{x_t - x_t^*}^2 \leq{}& \frac{6 C^6}{\decayvar^2(1 - \decayvar)(1 - \decayvar^2)^2}\cdot \decayvar^{4k} \left(\sum_{i=1}^{t-1}\lambda^{i-1} \norm{x_{t-i} - x_{t-i}^*}^2\right)\nonumber\\
    &+ \frac{6 C^4 (C + 1)^2}{\lambda^4(1 - \lambda)(1 - \lambda^2)^2}\cdot \lambda^{2k} \sum_{i=1}^T \lambda^{\abs{i - t-k+1}}\norm{x_i^*}^2 + 3 C^2 \decayvar^{2(T - t)} \norm{x_T^*}^2\label{thm:new-distance-MPC-infty:e4:s1}\\
    \leq{}& \frac{6 C^6}{\decayvar^2(1 - \decayvar)(1 - \decayvar^2)^2}\cdot \decayvar^{4k} \left(\sum_{i=1}^{t-1}\lambda^{i-1} \norm{x_{t-i} - x_{t-i}^*}^2\right)\nonumber\\
    &+ \frac{12 C^4 (C + 1)^2}{\lambda^4 (1 - \lambda)(1 - \lambda^2)^2 m_f}\cdot \lambda^{2k} \sum_{i=1}^T \lambda^{\abs{i - t - k + 1}}H_i^* + \frac{6 C^2}{m_f}\cdot \decayvar^{2(T - t)} H_T^*\label{thm:new-distance-MPC-infty:e4:s2},
\end{align}
\end{subequations}
where we used the inequality that for any real numbers $r_1, r_2, r_3$,
\[(r_1 + r_2 + r_3)^2 \leq 3 (r_1^2 + r_2^2 + r_3^2)\]
in \eqref{thm:new-distance-MPC-infty:e4:s1}, and the strong convexity of state costs in \eqref{thm:new-distance-MPC-infty:e4:s2}.

Summing inequality \eqref{thm:new-distance-MPC-infty:e4} over $t = 1, \ldots, T - k$ gives
\begin{align*}
    &\sum_{t=1}^{T - k} \norm{x_t - x_t^*}^2 - \sum_{t=1}^{T-k} \frac{6 C^6}{\decayvar^2(1 - \decayvar)(1 - \decayvar^2)^2}\cdot \decayvar^{4k} \left(\sum_{i=1}^{t-1}\lambda^{i-1} \norm{x_{t-i} - x_{t-i}^*}^2\right)\\
    \leq{}& \frac{12 C^4 (C + 1)^2}{\lambda^4 (1 - \lambda)(1 - \lambda^2)^2 m_f}\cdot \lambda^{2k} \sum_{t=1}^{T-k} \sum_{i=1}^T \lambda^{\abs{i - t-k+1}}H_i^* + \frac{6 C^2}{m_f}\cdot \sum_{t=1}^{T - k} \decayvar^{2(T - t)} H_T^*
\end{align*}

When $C \ge 1$, the right-hand side becomes: 
\begin{align*}
    \mathrm{RHS}
    \leq{}& \frac{12 C^4 (C + 1)^2}{\lambda^4 (1 - \lambda)(1 - \lambda^2)^2 m_f}\cdot \lambda^{2k} \sum_{i=1}^T \sum_{t=1}^{T-k}\lambda^{\abs{i - t-k+1}}H_i^* + \frac{6 C^2}{(1 - \lambda^2) m_f}\cdot \lambda^{2k} H_T^*\nonumber\\
    \leq{}& \frac{12 C^4 (C + 1)^2}{\lambda^4(1 - \lambda)(1 - \lambda^2)^2 m_f}\cdot \lambda^{2k} \left(\frac{H_T^*}{1 - \lambda} + \sum_{i=1}^{T-1} \frac{2 H_i^*}{1 - \lambda}\right) + \frac{6 C^2}{(1 - \lambda^2) m_f}\cdot \lambda^{2k} H_T^*\nonumber\\
    \leq{}& \frac{24 C^4 (C + 1)^2}{\lambda^4(1 - \lambda)^2(1 - \lambda^2)^2 m_f}\cdot \lambda^{2k} \sum_{i=1}^T H_i^*,
\end{align*}

For the left-hand side, we re-arrange the sum and  use the assumption that $k$ is sufficiently large.
\begin{align*}
    \mathrm{LHS}
    &= \sum_{t=1}^{T - k} \norm{x_t - x_t^*}^2 - \frac{6 C^6}{\decayvar^2(1 - \decayvar)(1 - \decayvar^2)^2}\cdot \decayvar^{4k} \sum_{t=1}^{T-k-1} \sum_{j=t+1}^{T-k}\lambda^{j-t-1} \norm{x_{t} - x_{t}^*}^2 \\
    &\ge \sum_{t=1}^{T - k} \norm{x_t - x_t^*}^2 - \frac{6 C^6}{\decayvar^2(1 - \decayvar)(1 - \decayvar^2)^2}\cdot \decayvar^{4k} \sum_{t=1}^{T-k} \frac{1}{1 - \lambda} \norm{x_{t} - x_{t}^*}^2 \\
    &\ge \sum_{t=1}^{T - k} \norm{x_t - x_t^*}^2 - (1 - \varepsilon) \sum_{t=1}^{T - k} \norm{x_t - x_t^*}^2 \\
    &= \varepsilon \sum_{t=1}^{T - k} \norm{x_t - x_t^*}^2
\end{align*}
Therefore, we we have the inequality
\begin{equation}\label{thm:new-distance-MPC-infty:e5}
    \sum_{t=1}^{T - k} \norm{x_t - x_t^*}^2 \le \varepsilon^{-1} \frac{24 C^4 (C + 1)^2}{\lambda^4(1 - \lambda)^2(1 - \lambda^2)^2 m_f}\cdot \lambda^{2k} \sum_{i=1}^T H_i^*.
\end{equation}

Since
\[(u_t - {u}_t^*) = B_t^\dagger \left((x_{t+1} - {x}_{t+1}^*) - A_t (x_t - {x}_t^*)\right),\]
we have
\[\norm{u_t - {u}_t^*} \leq b'\left(\norm{x_{t+1} - {x}_{t+1}^*} + a \norm{x_t - {x}_t^*}\right).\]
Therefore, by Corollary \ref{coro:smooth-cost}, for any $\eta > 0$ we have
\begin{align}\label{thm:new-distance-MPC-infty:e6}
    &\iota_{t}^1(x_{t}, x_{t+1}) - (1 + \eta)\iota_{t}^1({x}_{t}^*, {x}_{t+1}^*)\nonumber\\
    ={}& \left(f_{t+1}(x_{t+1}) - (1 + \eta) f_{t+1}({x}_{t+1}^*)\right) + \left(c_{t+1}(u_{t}) - (1 + \eta) c_{t+1}({u}_{t}^*)\right)\nonumber\\
    \leq{}& \frac{1}{2}\left(1 + \frac{1}{\eta}\right)\left(\ell_f \norm{x_{t+1} - {x}_{t+1}^*}^2 + \ell_c \norm{u_t - {u}_t^*}^2\right)\nonumber\\
    \leq{}& \frac{1}{2}\left(1 + \frac{1}{\eta}\right) \left(\ell_f + 2(b')^2 \ell_c\right)\norm{x_{t+1} - {x}_{t+1}^*}^2 + \frac{1}{2}\left(1 + \frac{1}{\eta}\right) 2 a^2 (b')^2 \ell_c \norm{x_t - {x}_t^*}^2\nonumber\\
    \leq{}& \left(1 + \frac{1}{\eta}\right)\cdot \frac{L_4}{2} \left(\norm{x_{t} - {x}_{t}^*}^2 + \norm{x_{t+1} - {x}_{t+1}^*}^2\right),
\end{align}
where
\[L_4 := \ell_f + 2(b')^2 \ell_c + 2 a^2 (b')^2 \ell_c.\]

Then, for any $\eta > 0$, we obtain the following inequality:
\begin{subequations}\label{equ:new_regret_e1}
\begin{align}
    &cost(PC_k) - (1 + \eta) cost(OPT)\nonumber\\
    ={}& \left(\sum_{t = 0}^{T-k-1} \iota_{t}^1(x_{t}, x_{t + 1}) + \iota_{T-k}^k(x_{T-k}, x_T)\right) - (1 + \eta)\left(\sum_{t = 0}^{T-k-1} \iota_{t}^1({x}_{t}^*, {x}_{t + 1}^*) + \iota_{T-k}^k(x_{T-k}^*, x_T^*)\right)\nonumber\\
    ={}& \sum_{t = 0}^{T-k-1} \left(\iota_{t}^1(x_{t}, x_{t+1}) - (1 + \eta)\iota_{t}^1(x_{t}^*, x_{t+1}^*)\right) + \left(\iota_{T-k}^k(x_{T-k}, x_T) - (1 + \eta)\iota_{T-k}^k(x_{T-k}^*, x_T^*)\right) \nonumber\\
    \leq{}& \sum_{t = 0}^{T-k-1} \left(\iota_{t}^1(x_{t}, x_{t+1}) - (1 + \eta)\iota_{t}^1(x_{t}^*, x_{t+1}^*)\right) + \left(\iota_{T-k}^k(x_{T-k}, x_T^*) - (1 + \eta)\iota_{T-k}^k(x_{T-k}^*, x_T^*)\right) \label{equ:new_regret_e1:s1}\\
    \leq{}& \left(1 + \frac{1}{\eta}\right)\cdot \frac{L_4}{2} \sum_{t = 0}^{T - k - 1} \left(\norm{x_{t} - x_{t}^*}^2 + \norm{x_{t+1} - x_{t+1}^*}^2\right) + \left(1 + \frac{1}{\eta}\right)\cdot \frac{L_0 + \ell_f}{2}\norm{x_{T-k} - x_{T-k}^*}^2 \label{equ:new_regret_e1:s2}\\
    \le{}& \left(1 + \frac{1}{\eta}\right)\cdot L_4 \sum_{t = 0}^{T - k - 1} \norm{x_{t} - x_{t}^*}^2 + \left(1 + \frac{1}{\eta}\right)\cdot \frac{L_4 + L_0 + \ell_f}{2}\norm{x_{T-k} - x_{T-k}^*}^2\nonumber\\
    \leq{}& \left(1 + \frac{1}{\eta}\right)\cdot \frac{2L_4 + L_0 + \ell_f}{2} \cdot \sum_{t = 1}^{T - k} \norm{x_{t} - x_{t}^*}^2\nonumber\\
    \leq{}& \left(1 + \frac{1}{\eta}\right)\cdot \frac{1}{\varepsilon} \cdot \frac{12 C^4 (C + 1)^2 (2L_4 + L_0 + \ell_f)}{\lambda^4 (1 - \lambda)^2(1 - \lambda^2)^2 m_f}\cdot \lambda^{2k} \cdot cost(OPT), \label{equ:new_regret_e1:s3}
\end{align}
\end{subequations}
where we use the fact that our algorithm $PC_k$ plans optimally after time step $T-k$ in \eqref{equ:new_regret_e1:s1}; we also use \eqref{thm:new-distance-MPC-infty:e6} and Corollary \ref{coro:smooth-cost} in \eqref{equ:new_regret_e1:s2}, and \eqref{thm:new-distance-MPC-infty:e5} in \eqref{equ:new_regret_e1:s3}.

Set $\eta = \lambda^{k}$ in the above inequality gives that
\begin{align*}
    cost(PC_k) &\leq \left(1 + \lambda^k \cdot \left(1 + (1 + \lambda^k)\cdot \frac{1}{\varepsilon} \cdot \frac{12 C^4 (C + 1)^2 (2L_4 + L_0 + \ell_f)}{\lambda^4 (1 - \lambda)^2(1 - \lambda^2)^2 m_f}\right)\right) \cdot cost(OPT)\\
    &\leq \left(1 + \lambda^k \cdot \left(1 + \frac{24 C^4 (C + 1)^2 (2L_4 + L_0 + \ell_f)}{\epsilon \cdot \lambda^4 (1 - \lambda)^2(1 - \lambda^2)^2 m_f}\right)\right) \cdot cost(OPT).
\end{align*}
\end{proof}

\end{document}